\documentclass[11pt,letterpaper,reqno]{amsart}

\usepackage{latexsym,amssymb}
\usepackage{amsmath}
\usepackage{amsthm} 
\usepackage{amsfonts}
\usepackage[mathscr]{eucal}
\usepackage{mathrsfs}
\usepackage{color}
\usepackage[abbrev]{amsrefs}
\usepackage[T1]{fontenc}

\theoremstyle{plain}
\newtheorem{Theorem}{\bf Theorem}[section]
\newtheorem{Lemma}{\bf Lemma}[section]
\newtheorem{Proposition}{\bf Proposition}[section]
\newtheorem{Corollary}{\bf Corollary}[section]
\newtheorem{Remark}{\bf Remark}[section]
\newtheorem{Example}{\bf Example}[section]
\newtheorem{Definition}{\bf Definition}[section]

\newtheorem{thm}{Theorem}
\newtheorem*{thm*}{Theorem}

\newenvironment{theorem}{\begin{Theorem}$\!\!\!$}{\end{Theorem}}
\newenvironment{lemma}{\begin{Lemma}$\!\!\!$}{\end{Lemma}}
\newenvironment{proposition}{\begin{Proposition}$\!\!\!$}{\end{Proposition}}
\newenvironment{corollary}{\begin{Corollary}$\!\!\!$}{\end{Corollary}}
\newenvironment{remark}{\begin{Remark}$\!\!\!$}{\end{Remark}}

\newenvironment{definition}{\begin{Definition}$\!\!\!$}{\end{Definition}}

\numberwithin{equation}{section}

\newcommand{\R}{\mathbb{R}}

\newcommand{\K}{\mathsf{K}}

\newcommand{\G}{\mathsf{G}}

\newcommand{\dist}{d}
\DeclareMathOperator{\Dist}{dist}

\def\XXint#1#2#3{{\setbox0=\hbox{$#1{#2#3}{\int}$}
\vcenter{\hbox{$#2#3$}}\kern-.5\wd0}}

\begin{document}

\title[fractional semilinear heat equation]
{Optimal singularities of initial data of \\
a fractional semilinear heat equation in  open sets}

\author{Kotaro Hisa}
\address{Mathematical Institute, Tohoku University, 6-3 Aoba, Aramaki, Aoba-ku, Sendai 980-8578, Japan.}
\email{kotaro.hisa.d5@tohoku.ac.jp}
\thanks{}


\date{}
\dedicatory{}

\begin{abstract}
We consider necessary conditions and sufficient conditions on the solvability of the  Cauchy--Dirichlet problem for a fractional semilinear heat equation in  open sets (possibly unbounded and disconnected) with a smooth boundary.
Our conditions enable us to identify the optimal strength of the admissible singularity of initial data  for the local-in-time solvability and they differ in the interior of the set and on the boundary of the set.
\end{abstract}

\maketitle


\section{Introduction}

\subsection{Introduction.}

This paper is concerned with 
the local-in-time solvability of the  Cauchy--Dirichlet problem for 
\begin{equation}
\label{eq:fjt}
\left\{
\begin{array}{ll}
	\partial_t u +(-\Delta)^\frac{\theta}{2}|_\Omega u = u^p, \quad  &x\in \Omega,\,\,\, t>0, \vspace{3pt}\\
	u = 0, \quad  &x\in  \mathbb{R}^N\setminus\Omega,\,\,\, t>0, \vspace{3pt}
\end{array}
\right.
\end{equation}
where  $\Omega$ is a  open set in ${\mathbb R}^N$ (possibly unbounded and disconnected) with a nonempty $C^{1,1}$ boundary, $N\ge1$, $p>1$, and $0<\theta<2$.
In this paper all solutions of \eqref{eq:fjt} are assumed to be nonnegative.
For $0<\theta<2$, the fractional Laplacian $(-\Delta)^{\theta/2}$ can be written in the form
\[
-(-\Delta)^\frac{\theta}{2} u(x) = c \lim_{\epsilon \to +0} \int_{\{y\in \mathbb{R}^N; |x-y|>\epsilon\}}  \frac{u(y)-u(x)}{|x-y|^{N+\theta}} \, dy
\]
for some specific constant $c=c(N,\theta)>0$.
Furthermore, $(-\Delta)^{\theta/2}|_\Omega$ denotes the fractional Laplacian with zero exterior condition.
For more details, see, for example, \cite{DPV}, which  summarizes many properties of the fractional Laplacian $(-\Delta)^{\theta/2}$.
Throughout this paper, we denote
\[
p_\alpha(d,l) := 1+ \frac{\alpha}{d+l}
\]
for $\alpha>0$, $d\ge1$, and $l\ge0$.
For any $x\in \overline{\Omega}$ and $r>0$, set
\begin{equation*}
\begin{split}
B(x,r) := \{y\in \mathbb{R}^N; |x-y|<r\}, \quad 
B_\Omega(x,r) := B(x,r) \cap \overline{\Omega}.
\end{split}
\end{equation*}
For a Borel set $A\subset \mathbb{R}^N$, $\chi_A(x)$ denotes the characteristic function of $A$.

The solvability of the Cauchy--Dirichlet problem for \eqref{eq:fjt} (including the case of $\theta\ge2$ and the case of $\Omega = \mathbb{R}^N$) has been studied in many papers. See, for example, \cite{BC, BP, CCV, F, HI01, HIT2, IKO1, IKO2, IS, KY, LS, L, QS, RS,  S, T, TW, W1, W2, XTS, Z} and references therein.
Of course, in the case of $\Omega=\mathbb{R}^N$, we ignore the boundary condition, and in the case where $\theta$ is a positive even integer, $\mathbb{R}^N\setminus\Omega$ in the boundary condition is replaced by $\partial \Omega$.
Among them, the author of this paper, Ishige, and Takahashi \cite{HIT2} considered the solvability of  the Cauchy--Dirichlet problem for 
\begin{equation}
\label{eq:HS}
\left\{
\begin{array}{ll}
	\partial_t u -\Delta u = u^p, \quad  &x\in \Omega,\,\,\, t>0, \vspace{3pt}\\
	u = 0, \quad  &x\in  \partial\Omega,\,\,\, t>0, \vspace{3pt}\\
\end{array}
\right.
\end{equation}
where $N\ge1$, $p>1$, and 
\begin{equation*}
\Omega =\mathbb{R}^N_+:=
\left\{
\begin{array}{ll}
	\mathbb{R}^{N-1}\times(0,\infty) \quad  &\mbox{if} \quad N\ge2, \vspace{3pt}\\
	(0,\infty) \quad   &\mbox{if} \quad N=1. \vspace{3pt}\\
\end{array}
\right.
\end{equation*}
For $d=1,2,\cdots$, let $g_d$ be the heat kernel in $\mathbb{R}^d\times(0,\infty)$, that is,
\[
g_d(x,t) := \frac{1}{(4\pi t)^\frac{d}{2}} \exp\left(-\frac{|x-y|^2}{4t}\right)
\]
for $x\in \mathbb{R}^d$ and $t\in(0,\infty)$.
Let $p=p(x,y,t)$ be the Dirichlet heat kernel in $\Omega\times(0,\infty)$, that is,
\[
p(x,y,t) := g_{N-1}(x'-y',t)[g_1(x_N-y_N,t) - g_1(x_N+y_N,t)] 
\]
for $x=(x',x_N)$, $y=(y',y_N)\in \overline{\Omega}$, and $t>0$.
The Cauchy--Dirichlet problem for \eqref{eq:HS} can possess a solution even if $u(\cdot,0)$ is not Radon measure on $\overline{\Omega}$
due to the boundary condition.
For example, Tayachi and Weissler \cite{TW} proved that if $1<p<p_2(N,1)$ and $u(\cdot,0)$ satisfies
\begin{equation}
\label{eq:TW}
u(\cdot,0) = -\kappa\partial_{x_N} \delta_N \quad \mbox{on} \quad \overline{\Omega} 
\end{equation}
for  sufficiently small $\kappa>0$,
then problem \eqref{eq:HS} with \eqref{eq:TW} possesses a local-in-time solution, where $\delta_N$ is the $N$-dimensional Dirac measure concentrated at the origin.
For this reason, we could not treat  initial data of  \eqref{eq:HS} in the framework of the Radon measure on $\overline{\Omega}$.
In order to solve this problem the authors of \cite{HIT2} introduced the following idea.
By the explicit formula of $p(x,y,t)$ we see that for $y=(y',0)\in\partial \Omega$,
\[
0< \lim_{y_N\to+0} \frac{p(x,y',y_N,t)}{y_N} = \partial_{y_N} p(x,y',0,t) <\infty
\]
for $x\in \Omega$ and $t\in(0,\infty)$.
Therefore, the function
\[
k(x,y,t) := 
\left\{
\begin{array}{ll}
	\displaystyle{\frac{p(x,y,t)}{y_N}} \quad  &\mbox{if} \quad (x,y,t) \in \overline{\Omega}\times \Omega \times(0,\infty), \vspace{3pt}\\
	\partial_{y_N} p(x,y,t) \quad  &\mbox{if} \quad (x,y,t) \in \overline{\Omega}\times \partial\Omega \times(0,\infty),  \vspace{3pt}\\
\end{array}
\right.
\]
is well-defined and continuous on $\overline{\Omega}\times\overline{\Omega}\times(0,\infty)$.
Using this function, the solution of the heat equation on $\Omega\times(0,\infty)$ can be rewritten as 
\[
\int_{\overline{\Omega}} p(x,y,t) u(y,0)\,dy = \int_{\overline{\Omega}} k(x,y,t) y_Nu(y,0)\,dy.
\]
Since $k(x,y,t)$ is positive and finite for $(x,y,t)\in \Omega\times\overline{\Omega}\times(0,\infty)$,
they gave an initial condition of Radon measure on $\overline{\Omega}$ to $x_N u(\cdot,0)$, instead of $u(\cdot,0)$ itself, that is,
\begin{equation}
\label{eq:HSini}
x_N u(\cdot,0) = \mu \quad\mbox{on}\quad \overline{\Omega}, 
\end{equation}
where $\mu$ is a Radon measure on $\overline{\Omega}$.
Thanks to this idea, we can treat the initial condition of  \eqref{eq:HS} in the framework of the Radon measure.
In additions, they obtained sharp necessary conditions and sufficient conditions on the solvability of  problem \eqref{eq:HS} with \eqref{eq:HSini}.
Applying these conditions, for any $z\in \overline{\Omega}$, they found a nonnegative measurable function $f_z$ on $\Omega$
with the following properties:
\begin{itemize}
\item there exists $R>0$ such that $f_z$ is continuous in $B_\Omega(z,R)\setminus\{z\}$ and $f_z=0$ outside $B_\Omega(z,R)$;
\item there exists $\kappa_z>0$ such that problem \eqref{eq:HS} with $\mu = \kappa x_N f_z(x)$, possesses a local-in-time solution if $0<\kappa<\kappa_z$ and it possesses no local-in-time solutions if $\kappa>\kappa_z$.
\end{itemize}
They termed the singularity of the function $f_z$ at $x=z$ an \textit{optimal singularity} of initial data for the solvability of problem \eqref{eq:HS} with \eqref{eq:HSini} at $x=z$.
In Theorem A they identified optimal singularities of initial data in the interior of $\Omega$.
\begin{thm}
Let $z\in \Omega$. Set
\[
f_z(x):= 
\left\{
\begin{array}{ll}
	|x-z|^{-\frac{2}{p-1}}\chi_{B_\Omega(z,1)}(x) \quad  &\mbox{if} \quad p>p_2(N,0), \vspace{3pt}\\
	|x-z|^{-N}|\log|x-z||^{-\frac{N}{2}-1}\chi_{B_\Omega(z,1/2)}(x)\quad  &\mbox{if} \quad p=p_2(N,0),  \vspace{3pt}\\
\end{array}
\right.
\]
for $x\in \Omega$. Then there exists $\kappa_z>0$ with the following properties:
\begin{itemize}
\item[(i)] problem \eqref{eq:HS} with \eqref{eq:HSini} possesses a local-in-time solution with $\mu= \kappa x_Nf_z(x)$ if $0<\kappa<\kappa_z$;
\item[(ii)] problem \eqref{eq:HS} with \eqref{eq:HSini} possesses no local-in-time solutions with $\mu= \kappa x_Nf_z(x)$ if $\kappa>\kappa_z$.
\end{itemize}
Here $\sup_{z\in \Omega}\kappa_z<\infty$.
\end{thm}
In Theorem B they identified optimal singularities of initial data on the boundary.
Due to the boundary condition, optimal singularities are stronger than those in Theorem A.
\begin{thm}
Set
\[
f(x):= 
\left\{
\begin{array}{ll}
	|x|^{-\frac{2}{p-1}}\chi_{B_\Omega(0,1)}(x) \quad  &\mbox{if} \quad p>p_2(N,1), \vspace{3pt}\\
	|x|^{-N-1}|\log|x||^{-\frac{N+1}{2}-1}\chi_{B_\Omega(0,1/2)}(x)\quad  &\mbox{if} \quad p=p_2(N,1),  \vspace{3pt}\\
\end{array}
\right.
\]
for $x\in \Omega$. Then there exists $\kappa_0>0$ with the following properties:
\begin{itemize}
\item[(i)] problem \eqref{eq:HS} with \eqref{eq:HSini} possesses a local-in-time solution with $\mu= \kappa x_Nf(x)$ if $0<\kappa<\kappa_0$;
\item[(ii)] problem \eqref{eq:HS} with \eqref{eq:HSini} possesses no local-in-time solutions with $\mu= \kappa x_Nf(x)$ if $\kappa>\kappa_0$.
\end{itemize}
\end{thm}

We go back to  \eqref{eq:fjt}.
In this case, though no explicit formulas of the Dirichlet heat kernel have been obtained,  two-sided estimates of it have been obtained (see Theorem C below).
In this paper, we give necessary  conditions and sufficient conditions for the local-in-time  solvability of the Cauchy--Dirichlet problem for \eqref{eq:fjt} by using them.
Furthermore, applying these conditions, we identify optimal singularities of initial data.
Our arguments are basically based on \cite{HIT2}.

\subsection{Notation and the definition of solutions.}
In order to state our main results, we introduce some notation and formulate the definition of solutions.
We denote by $\mathcal{M}$ the set of nonnegative Radon measures on $\overline{\Omega}$.
For any $L^1_{loc}(\overline{\Omega})$-function $\mu$, we often identify $d\mu=\mu(x)\,dx$ in $\mathcal{M}$.
For any $T\in(0,\infty]$, we set $Q_T:=\Omega\times (0,T)$.
For two nonnegative functions $f$ and $g$, the notation $f\asymp g$ means that there exist positive constants $c_1$ and $c_2$ such that $c_1g(x)\le f(x)\le c_2 g(x)$ in the common domain of definition of $f$ and $g$.
For $a,b\in {\mathbb R}$, $a\wedge b := \min\{a,b\}$.

Let $\Gamma_\theta =\Gamma_\theta(x,t)$ be the fundamental solution of 
\[
\partial_t v + (-\Delta)^\frac{\theta}{2} v = 0 \quad \mbox{in} \quad \mathbb{R}^N\times (0,\infty),
\]
where $N\ge1$ and $0<\theta<2$.
For any $x,y\in \overline{\Omega}$ and $t>0$, let $G=G(x,y,t)$ be the Dirichlet heat kernel on $\Omega$.
Then, $G$ is continuous on $\overline{\Omega}\times\overline{\Omega}\times(0,\infty)$ and  satisfies
\begin{align}
&\label{eq:0.2} \int_{\Omega} G(x,y,t) \, dy \le1,\\
&\label{eq:0.3}\int_\Omega G(x,z,t) G(z,y,s) \,dz = G(x,y,t+s),
\end{align}
for all $x,y\in \Omega$ and $s,t >0$,
 and 
\begin{equation*}
	\left\{
	\begin{aligned}
	&G(x,y,t) = G(y,x,t)\quad && \mbox{if} \quad (x,y,t)\in \mathbb{R}^N\times \mathbb{R}^N\times(0,\infty),\\
	&G(x,y,t) >0 \quad && \mbox{if} \quad (x,y,t)\in {\Omega}\times{\Omega}\times(0,\infty),\\
	&G(x,y,t) =0 \quad && \mbox{if} \quad (x,y,t)\in {\Omega}^c\times\mathbb{R}^N\times(0,\infty).
	\end{aligned}
	\right.
\end{equation*}
See \cite{CCV}.
Furthermore, Chen, Kim, and Song \cite{CKS} obtained the following two-sided estimates  of $G$:
\begin{thm}
Let $\Omega$ be a open set in $\mathbb{R}^N$ with a $C^{1,1}$ boundary  and $d(x) := \Dist(x,\partial\Omega)$.
\begin{itemize}
\item[(i)] There exists $T'>0$  depending only on $\Omega$ such that
\begin{equation}
\label{eq:1.2}
\begin{split}
G(x,y,t) \asymp\left(1\wedge \frac{d(x)^\frac{\theta}{2}}{\sqrt{t}}\right)\left(1\wedge \frac{d(y)^\frac{\theta}{2}}{\sqrt{t}}\right)\Gamma_\theta(x-y,t)
\end{split}
\end{equation}
for all $x,y\in\overline{\Omega}$ and $t\in(0,T']$.
\item[(ii)]When $\Omega$ is bounded and $t>T'$, one has
\[
 G(x,y,t) \asymp  d(x)^\frac{\theta}{2}d(y)^\frac{\theta}{2} e^{-\lambda_1 t}
\]
for all $x,y \in \overline{\Omega}$ and $t>T'$. Here, $\lambda_1>0$ is the smallest eigenvalue of the Dirichlet fractional Laplacian $(-\Delta)^{{\theta/2}}|_\Omega$.
\end{itemize}
\end{thm}
See also \cite{BGR, CT,CCV}.
For $x\in\overline{\Omega}$, $y\in\partial \Omega$, and $t>0$, define the $\theta/2$-normal derivative as
\[
D_{\frac{\theta}{2}} G(x,y,t) := \lim_{\tilde{y}\in \Omega, \tilde{y}\to y}\frac{G(x,\tilde{y},t)}{\dist(\tilde{y})^\frac{\theta}{2}},
\]
and in virtue of  the result in \cite{CCV},  this limit exists for all $x\in \overline{\Omega}$, $y\in \partial\Omega$, and $t>0$.
Define
\begin{equation*}
K(x,y,t):= \left\{
	\begin{aligned}
	&\frac{G(x,y,t)}{\dist(y)^\frac{\theta}{2}}\quad && \mbox{if} \quad (x,y,t)\in \overline{\Omega}\times{\Omega}\times(0,\infty),\\
	&D_{\frac{\theta}{2}} G(x,y,t) \quad && \mbox{if} \quad (x,y,t)\in \overline{\Omega}\times \partial{\Omega}\times(0,\infty).\\
	\end{aligned}
	\right.
\end{equation*}
Then $K\in C(\overline{\Omega}\times\overline{\Omega}\times(0,\infty))$ and 
\begin{equation*}
	\left\{
	\begin{aligned}
	&K(x,y,t) >0 \quad && \mbox{if} \quad (x,y,t)\in {\Omega}\times\overline{\Omega}\times(0,\infty),\\
	&K(x,y,t) =0 \quad && \mbox{if} \quad (x,y,t)\in \partial{\Omega}\times\overline{\Omega}\times(0,\infty).
	\end{aligned}
	\right.
\end{equation*}
Furthermore, it follows from \eqref{eq:1.2} that $K$ satisfies
\begin{equation}
\label{eq:1.3}
K(x,y,t) \asymp \left(1\wedge \frac{d(x)^\frac{\theta}{2}}{\sqrt{t}}\right)\left(\frac{1}{d(y)^\frac{\theta}{2}}\wedge \frac{1}{\sqrt{t}}\right)\Gamma_\theta(x-y,t)
\end{equation}
for all $x,y \in \overline{\Omega}$ and $t\in(0,T']$.
From the analogy of the result \cite{HIT2}, we give an initial condition to $d(x)^{\theta/2}u(\cdot,0)$, instead of $u(\cdot,0)$.
Namely, this paper is concerned with the solvability of the Cauchy--Dirichlet problem
\begin{equation*}
\label{eq:SHE}
{\rm (SHE)} \quad
	\left\{
	\begin{aligned}
	&\partial_t u +(-\Delta)^\frac{\theta}{2}|_\Omega u =  u^p,\quad && x\in \Omega,\,\, t\in(0,T), \\
	&u = 0, && x\in \mathbb{R}^N\setminus\Omega,\,\, t\in(0,T), \\
	&\dist(x)^\frac{\theta}{2} u(0) =  \mu\quad && \mbox{in} \quad \overline{\Omega}, \\
	\end{aligned}
	\right.
\end{equation*}
where  $N\geq1$, $0<\theta<2$, $p>1$,  $T>0$,  and
$\mu$ is a nonnegative Radon measure on $\overline{\Omega}$.

Next, we formulate the definition of solutions of {\rm (SHE)}.
\begin{definition}
\label{Def:1.1}
Let $u$ be a nonnegative measurable function in $\Omega\times (0,T)$, where $0<T\le \infty$.
We say that $u$ is a solution 
of {\rm (SHE)} in $Q_T$ if u satisfies
\begin{equation}
\label{Def:1.1.2}
\begin{split}
\infty> u(x,t) 
&= \int_{\overline{\Omega}} K(x,y,t) \, d\mu(y) +\int_0^t\int_{\Omega} G(x,y,t-s)u(y,s)^p  \,dyds
\end{split}
\end{equation}
for almost all $(x,t) \in Q_T$.
If $u$ satisfies the above equality with $=$ replaced by $\ge$, then $u$ is said to be a supersolution of {\rm (SHE)} in $(x,t) \in Q_T$.
\end{definition}

\subsection{Main results.}
Now we are ready to state our main results of this paper. 
Throughout of this paper, denote $T_*>0$ by
\[
T_*:= T' \wedge \frac{(\mbox{diam}\, \Omega)^\theta}{16}.
\]
In the first theorem, we identify the optimal singularities in the interior of $\Omega$ of initial data of the solvability of problem {\rm (SHE)}.

\begin{theorem}
\label{Theorem:1.1}
Let $z\in \Omega$. Set 
\begin{equation*}
\varphi_z(x):= \left\{
\begin{array}{ll}
	|x-z|^{-\frac{\theta}{p-1}}\chi_{B_\Omega(z,1)}(x), \quad  &\mbox{if}\quad p>p_\theta(N,0), \vspace{3pt}\\
	|x-z|^{-N}|\log |x-z||^{-\frac{N}{\theta}-1}\chi_{B_\Omega(z,1/2)}(x), \quad   &\mbox{if}\quad p=p_\theta(N,0), \vspace{3pt}
\end{array}
\right.
\end{equation*}
for $x\in \Omega$. Then there exists $\kappa_z>0$ with the following properties:
\begin{itemize}
\item[(i)] If $p<p_\theta(N,0)$, for any $\nu \in {\mathcal M}$ problem {\rm (SHE)} possesses a local-in-time solution with $\mu =d(x)^{{\theta/2}}\nu$;
\item[(ii)] problem {\rm (SHE)} possesses a local-in-time solution with $\mu = \kappa d(x)^{\theta/2}\varphi_z(x)$  if $\kappa < \kappa_z$;
\item[(iii)] problem {\rm (SHE)} possesses no local-in-time solutions with $\mu = \kappa d(x)^{\theta/2}\varphi_z(x)$ if $\kappa>\kappa_z$.
\end{itemize}
Here, $\sup_{z\in \Omega} \kappa_z<\infty$.
\end{theorem}

In the second theorem, we identify the optimal singularities on the boundary of $\Omega$ of initial data of the solvability of problem {\rm (SHE)}.

\begin{theorem}
\label{Theorem:1.2}
Let $z\in \partial \Omega$. Set 
\begin{equation*}
\psi_z(x):= \left\{
\begin{array}{ll}
	|x-z|^{-\frac{\theta}{p-1}}\chi_{B_\Omega(z,1)}(x), \quad  &\mbox{if}\quad p>p_\theta(N,\theta/2), \vspace{3pt}\\
	|x-z|^{-N-\frac{\theta}{2}}|\log |x-z||^{-\frac{2N+\theta}{2\theta}-1}\chi_{B_\Omega(z,1/2)}(x), \quad   &\mbox{if}\quad p=p_\theta(N,\theta/2), \vspace{3pt}
\end{array}
\right.
\end{equation*}
for $x\in \Omega$. Then there exists $\kappa_z>0$ with the following properties:
\begin{itemize}
\item[(i)] If $p<p_\theta(N,\theta/2)$, problem {\rm (SHE)} possesses a local-in-time solution for all $\mu\in {\mathcal M}$;
\item[(ii)] problem {\rm (SHE)} possesses a local-in-time solution with $\mu = \kappa d(x)^{\theta/2}\psi_z(x)$  if $\kappa < \kappa_z$;
\item[(iii)] problem {\rm (SHE)} possesses no local-in-time solutions with $\mu = \kappa d(x)^{\theta/2}\psi_z(x)$ if $\kappa>\kappa_z$.
\end{itemize}
Here, $\sup_{z\in \partial\Omega} \kappa_z<\infty$.
\end{theorem}

The rest of this paper is organized as follows. 
In Section 2 we collect some properties of the kernels $G$ and $K$ and prove some preliminary lemmas.
In Section 3 we obtain necessary conditions on the solvability of problem {\rm (SHE)}.
In Section 4 we obtain sufficient conditions on the solvability of problem {\rm (SHE)}.
In Section 5 by applying these conditions, we prove Theorems \ref{Theorem:1.1} and \ref{Theorem:1.2}.

\section{Preliminaries.}
In what follows we will use $C$ to denote generic positive constants. The letter $C$ may take different values within a calculation.
We first prove the following covering lemma.
\begin{lemma}
\label{Lemma:2.0}
Let $N\ge1$ and $\delta\in (0,1)$. Then there exists $m\in\{1,2,\cdots\}$ with the following properties.
\begin{itemize}
\item[(i)] For any $z\in \mathbb{R}^N$ and $r>0$, there exists $\{z_i\}_{i=1}^m\subset \mathbb{R}^N$ such that
\[
B(z,r) \subset \bigcup_{i=1}^m B(z_i,\delta r).
\]
\item[(ii)] For any $z\in\mathbb{R}^N$ and $r>0$, there exists $\{\overline{z}_i\}_{i=1}^m\subset B_\Omega(z,2r)$ such that
\[
B_\Omega (z,r) \subset \bigcup_{i=1}^m B_\Omega(\overline{z}_i,\delta r).
\]
\end{itemize}
\end{lemma}
\begin{proof}
Assertion (i) has been already proved in \cite{HIT2}. We prove assertion (ii).
We find $m\in\{1,2,\cdots\}$ and $\{\tilde{z}_i\}_{i=1}^{m}\subset B_\Omega(0,1)$ such that 
$B_\Omega (0,1) \subset \cup_{i=1}^{m}B_\Omega (\tilde{z}_i,\delta/2)$, so that
\begin{equation}
\label{eq:2.0}
B_\Omega (z,r) \subset \bigcup_{i=1}^{m} B_\Omega (z+r\tilde{z}_i,\delta r/2).
\end{equation}
Set $\overline{z}_i := z + r\tilde{z}_i$ if $z+r\tilde{z}_i\in \overline{\Omega}$
and
$\overline{z}_i\in B_\Omega(z+r\tilde{z}_i,\delta r/2)\cap \partial\Omega $ if $z+r\tilde{z}_i\not\in \overline{\Omega}$.
Then
\[
\overline{z}_i\in B_\Omega(z,2r),\quad
B_\Omega(z+r\tilde{z}_i,\delta r/2 )\subset B_\Omega(\overline{z}_i,\delta r) \quad \mbox{if} \quad B_\Omega(z+r \tilde{z}_i,\delta r/2) \neq \emptyset.
\]
This together with \eqref{eq:2.0} implies that 
\[
B_\Omega (z,r) \subset \bigcup_{i=1}^{m} B_\Omega (\overline{z}_i,\delta r).
\]
Then assertion (ii) follows, and the proof is complete.
\end{proof}

Next, we collect some properties of the kernels $G$ and $K$ and prepare preliminary lemmas.
We see that $\Gamma_\theta$ satisfies
\begin{align}
&\label{eq:2.1}
 \Gamma_\theta(x,t) \asymp  t^{-\frac{N}{\theta}}\wedge \frac{t}{|x|^{N+\theta}},\\
&\label{eq:2.2}
\int_{\mathbb{R}^N} \Gamma_\theta(x,t) \,dx =1,
\end{align}
for all $x\in\mathbb{R}^N$ and $t>0$ (see e.g., \cite{BGR, HI01}). 
Denote $D(x,t)$ by
\[
D(x,t) := \frac{d(x)^\frac{\theta}{2}}{d(x)^\frac{\theta}{2}+\sqrt{t}}
\]
for $(x,t)\in \overline{\Omega}\times(0,\infty)$.
Then the following lemmas hold.

\begin{lemma}
\label{Lemma:2.1}
\begin{itemize}
\item[(i)] There exists $C_1>0$ such that
\[
\int_{\R^N} \Gamma_\theta(x-y,t) \,dm_1(y) \le C_1 t^{-\frac{N}{\theta}} \sup_{z\in\R^N} m_1(B_\Omega(z,t^\frac{1}{\theta}))
\]
for all nonnegative Radon measure $m_1$ on $\R^N$ and $(x,t)\in \R^N\times(0,\infty)$.
\item[(ii)]  There exists $C_2>0$ such that
\begin{equation}
\label{eq:2.3}
K(x,y,t) \le C_2\frac{D(x,t)}{\dist(y)^\frac{\theta}{2}+\sqrt{t}} \Gamma_\theta(x-y,t)
\end{equation}
for all $(x,y,t)\in \overline{\Omega}\times\overline{\Omega}\times(0,T_*]$.
Furthermore, there exists $C_3>0$ such that
\begin{equation}
\label{eq:2.4}
\int_{\overline{\Omega}} \frac{K(x,y,t)}{D(x,t)}\,dm_2(y) \le C_3 t^{-\frac{N}{\theta}} \sup_{z\in\overline{\Omega}} \int_{B_\Omega(z,t^\frac{1}{\theta})} \frac{dm_2(y)}{\dist(y)^\frac{\theta}{2}+\sqrt{t}}
\end{equation} 
for all $m_2\in\mathcal{M}$ and $(x,t)\in \overline{\Omega}\times(0,T_*]$.
%
\end{itemize}
\end{lemma}
\begin{proof}
Assertion (i) follows from \cite[Lemma 2.1]{HI01}.
It follows that
\[
1\wedge ab\le (1\wedge a)(1\wedge b)(1+|a-b|)\le \frac{4ab(1+|a-b|)}{(1+a)(1+b)}
\]
for all $a,b>0$ (see e.g. \cite[Section 1.1]{MMZ}).
Let $t\in(0,T_*]$.
Then, by \eqref{eq:1.2}  we have
\begin{equation*}
\begin{split}
 G(x,y,t) 
&\le C \left(1\wedge \frac{d(x)^\frac{\theta}{2}}{\sqrt{t}}\right)\left(1\wedge \frac{d(y)^\frac{\theta}{2}}{\sqrt{t}}\right)\Gamma_\theta (x-y,t)\\
&\le \frac{Cd(x)^\frac{\theta}{2}d(y)^\frac{\theta}{2}}{(d(x)^\frac{\theta}{2}+\sqrt{t})(d(y)^\frac{\theta}{2}+\sqrt{t})}\Gamma_\theta (x-y,t)\\
& = CD(x,t)D(y,t)\Gamma_\theta (x-y,t)
\end{split}
\end{equation*}
for all $(x,y)\in \overline{\Omega}\times\overline{\Omega}$ and $t\in(0,T_*]$.
This implies that \eqref{eq:2.3} holds.
\eqref{eq:2.4} follows from \eqref{eq:2.3} and  assertion (i) with $m_1 = m_2\chi_{\overline{\Omega}}(y)/(d(y)^{\theta/2} + \sqrt{t})$.
\end{proof}

\begin{lemma}
\label{Lemma:2.2}
The integral kernels $G$ and $K$ satisfy
\begin{align}
&\label{eq:2.6}
\int_{\Omega} K(x,y,t)\,dx \le C_4t^{-\frac{1}{2}} \quad \mbox{for} \quad (y,t)\in \partial \Omega\times(0,T_*];\\
%
%
&\label{eq:2.7}
\int_{\partial \Omega} \frac{K(x,y,t)}{D(x,t)}\,d\sigma(y) \le C_5t^{-\frac{1}{2}-\frac{1}{\theta}} \quad \mbox{for} \quad (x,t)\in \overline{\Omega}\times(0,T_*];\\
&\label{eq:2.8}
\int_{\Omega} G(z,x,s) K(x,y,t) \,dx = K(z,y,t+s) \quad \mbox{for} \quad (z,y,t,s)\in \Omega\times \overline{\Omega} \times(0,\infty)^2,
\end{align}
where $C_4, C_5>0$ are constants depending only on $\Omega$, $N$, and $\theta$.
\end{lemma}
\begin{proof}
Let $y\in\partial \Omega$. By \eqref{eq:2.2} and \eqref{eq:2.3} we have
\begin{equation*}
\begin{split}
\int_{\Omega} K(x,y,t)\,dx 
&\le C\int_{\Omega}\frac{D(x,t)}{\dist(y)^\frac{\theta}{2}+\sqrt{t}} \Gamma_\theta(x-y,t)\,dx\\
&\le  Ct^{-\frac{1}{2}}\int_{\R^N}\Gamma_\theta(x-y,t)\,dx  = Ct^{-\frac{1}{2}}
\end{split}
\end{equation*}
for all $y\in\partial \Omega$ and $t\in(0,T_*]$.  Then \eqref{eq:2.6} follows.

By \eqref{eq:2.4} with $m_2 = 1\otimes \delta_1(\dist(x))$, we have
\begin{equation*}
\begin{split}
\int_{\partial \Omega} \frac{K(x,y,t)}{D(x,t)}\,d\sigma(y)
& =\int_{\overline{ \Omega}} \frac{K(x,y,t)}{D(x,t)}\,dm_2(y)\\
&\le C t^{-\frac{N}{\theta}} \sup_{z\in\overline{\Omega}} \int_{B_\Omega(z,t^\frac{1}{\theta})} \frac{dm_2(y)}{\dist(y)^\frac{\theta}{2}+\sqrt{t}}\\
&= C t^{-\frac{N}{\theta}-\frac{1}{2}} \sup_{z\in\overline{\Omega}} m_2(B_\Omega(z,t^\frac{1}{\theta})\cap \partial \Omega)\\
&\le Ct^{-\frac{1}{\theta}-\frac{1}{2}}
\end{split}
\end{equation*}
for all $(x,t)\in \overline{\Omega}\times(0,T_*]$,
where $\delta_1$ is the $1$-dimensional Dirac measure concentrated at the origin.
Then \eqref{eq:2.7} follows.

Let $y\in \Omega$. By \eqref{eq:0.3} we have
\begin{equation*}
\begin{split}
\int_{\Omega} G(z,x,s) K(x,y,t) \,dx 
&= \frac{1}{\dist(y)^\frac{\theta}{2}}\int_{\Omega} G(z,x,s) G(x,y,t) \,dx \\
&= \frac{G(z,y,t+s)}{\dist(y)^\frac{\theta}{2}} = K(z,y,t+s).
\end{split}
\end{equation*}
Let $y\in \partial \Omega$. By \eqref{eq:0.3}, \eqref{eq:1.2}, and  the dominated convergence theorem we have
\begin{equation*}
\begin{split}
\int_{\Omega} G(z,x,s) K(x,y,t) \,dx 
&= \int_{\Omega} G(z,x,s) \lim_{\tilde{y}\in \Omega,\tilde{y}\to y}\frac{G(x,\tilde{y},t)}{d(\tilde{y})^\frac{\theta}{2}} \,dx \\
&=  \lim_{\tilde{y}\in \Omega,\tilde{y}\to y}\frac{1}{d(\tilde{y})^\frac{\theta}{2}}\int_{\Omega} G(z,x,s)G(x,\tilde{y},t)\,dx \\
&=  \lim_{\tilde{y}\in \Omega,\tilde{y}\to y}\frac{G(z,\tilde{y},t+s)}{d(\tilde{y})^\frac{\theta}{2}}= K(z,y,t+s).
\end{split}
\end{equation*}
We obtain \eqref{eq:2.8} and the proof is complete.
\end{proof}

At the end of this section we prepare a lemma on an integral inequality.
This lemma has been already proved in \cite{HIT2}.
This idea of using this kind of lemma is due to \cite{LS}.
See also \cite{FHIL, HIT2, HK}.

\begin{lemma}
\label{Lemma:2.4}
Let  $\zeta$ be a nonnegative measurable function in $(0,T)$, where $T>0$.
Assume that 
\begin{equation*}
\infty > \zeta(t) \ge c_1 + c_2 \int_{t_*}^t s^{-\alpha} \zeta(s)^\beta \, ds \quad \mbox{for almost all} \quad t \in(t_*,T),
\end{equation*}
where $c_1,c_2>0$, $\alpha\ge0$, $\beta>1$, and $t_*\in (0,T/2)$.
Then there exists $C = C(\alpha,\beta)>0$ such that
\begin{equation*}
c_1 \le C c_1^{-\frac{1}{\beta-1}} t_*^\frac{\alpha-1}{\beta-1}.
\end{equation*}
In addition, if $\alpha=1$, then 
\begin{equation*}
c_1 \le (c_2(\beta-1))^{-\frac{1}{\beta-1}} \left[\log\frac{T}{2t_*}\right]^{-\frac{1}{\beta-1}}.
\end{equation*}
\end{lemma}

\section{Necessary conditions for the  local-in-time solvability.}

In this section we obtain necessary conditions on the local-in-time solvability of problem {\rm (SHE)}.
Our necessary conditions are as follows:

\begin{theorem}
\label{Theorem:3.1}
Let $N\ge1$, $0<\theta<2$, and $p>1$.
Assume problem~{\rm (SHE)} possesses a supersolution in $Q_T$,
where $T\in(0,T_*]$.
Then  there exists $\gamma_1=\gamma_1(\Omega, N, \theta, p)>0$ 
such that
\begin{equation}
\label{Thm:3.1.1}
\begin{split}
\mu(B_\Omega(z,\sigma))
 &\le 
\gamma_1 \sigma^{-\frac{\theta}{p-1}} \int_{B_\Omega(z,\sigma)} \dist(y)^\frac{\theta}{2}\,dy
\end{split}
\end{equation} 
for all $z\in \Omega$ and  $\sigma\in (0,T^{1/\theta})$. 
In addition,
\begin{itemize}
\item[(i)] if $p=p_\theta(N,0)$, then there exists  $\gamma_1'=\gamma_1'(\Omega, N, \theta)>0$ such that
\begin{equation}
\label{Thm:3.1.2}
\begin{split}
 \dist(z)^{-\frac{\theta}{2}}\mu(B_\Omega(z,\sigma))
 &\le 
\gamma_1' \left[\log\left(e+\frac{\sqrt{T}}{\sigma}\right)\right]^{-\frac{N}{\theta}}
\end{split}
\end{equation} 
for all $z\in\Omega$ with $d(z)\ge 3\sigma$ and $\sigma\in (0,T^{1/\theta})$.
\item[(ii)]  if $p=p_\theta(N,\theta/2)$, then there exists  $\gamma_1''=\gamma_1''(\Omega, N, \theta)>0$ such that
\begin{equation}
\label{Thm:3.1.3}
\mu(B_\Omega(z,\sigma))\le 
\gamma_1'' \left[\log\left(e+\frac{T^\frac{1}{\theta}}{\sigma}\right)\right]^{-\frac{2N+\theta}{2\theta}}
\end{equation} 
for all $z\in \partial\Omega$ and  $\sigma\in (0,T^{1/\theta})$. 
%
\end{itemize}
\end{theorem}
Compare with \cite{HIT2}.
In order to prove Theorem \ref{Theorem:3.1}, we first modify the arguments in \cite{HIT2} and prove Proposition \ref{Proposition:3.1} below.

\begin{proposition}
\label{Proposition:3.1}
Assume that there exists a supersolution of problem {\rm (SHE)} in $Q_T$, where $T\in (0,T_*]$.
Then there exists $\gamma>0$ such that 
\begin{equation}
\label{eq:3.1}
d(z)^{-\frac{\theta}{2}} \mu (B_\Omega (z,\sigma)) \le \gamma \sigma^{N-\frac{\theta}{p-1}}
\end{equation}
for all $z\in \Omega$ with $d(z) \ge T^{{1/\theta}}$ and $\sigma\in (0, T^{{1/\theta}}/16)$.
Furthermore, if $p=p_\theta(N,0)$, there exists $\gamma'>0$ such that
\begin{equation}
\label{eq:3.2}
d(z)^{-\frac{\theta}{2}} \mu (B_\Omega (z,\sigma)) \le \gamma'\left[\log\left(e+\frac{T^\frac{1}{\theta}}{\sigma}\right)\right]^{-\frac{N}{\theta}}
\end{equation} 
for all $z\in \Omega$ with $d(z)\ge T^{{1/\theta}}$ and $\sigma\in (0, T^{{1/\theta}}/16)$.
\end{proposition}

In order to prove Proposition \ref{Proposition:3.1}, we prepare two lemmas on the integral kernels.

\begin{lemma}
\label{Lemma:3.1}
For any $\epsilon\in (0,1/2)$, there exists $C>0$ such that
\[
\int_{B_\Omega (z,\sigma)} K(z,y,\sigma^\theta) \, d\mu(y) \ge C \sigma^{-N} d(z)^{-\frac{\theta}{2}} \mu (B_\Omega(z,\sigma))
\]
for all $\mu \in {\mathcal M }$, $z\in \Omega$ with $d(z)\ge T^{1/\theta}$,  $\sigma \in (0, \epsilon T^{1/\theta})$, and $T\in (0,T_*]$.
\end{lemma}
\begin{proof}
Let $\epsilon\in (0,1/2)$,  $\sigma \in (0, \epsilon T^{1/\theta})$, $z\in \Omega$ with $d(z)\ge T^{1/\theta}$, and $y\in B_\Omega(z,\sigma)$.
By \eqref{eq:1.3} and \eqref{eq:2.1} we have 
\begin{equation*}
\begin{split}
K(z,y,\sigma^\theta)\ge C \left(1\wedge \frac{d(z)^\frac{\theta}{2}}{\sigma^\frac{\theta}{2}}\right)\left(\frac{1}{d(y)^\frac{\theta}{2}} \wedge \frac{1}{\sigma^\frac{\theta}{2}}\right) 
\left(\sigma^{-N} \wedge \frac{\sigma^\theta}{|z-y|^{N+\theta}}\right).
\end{split}
\end{equation*}
Since 
\[
d(z)^\frac{\theta}{2} \ge T^\frac{1}{2}> \epsilon^{-\frac{\theta}{2}} \sigma^\frac{\theta}{2} > \sigma^\frac{\theta}{2},
\]

\[
d(y) > d(z) - \sigma \ge T^\frac{1}{\theta} - \sigma > (\epsilon^{-1}-1) \sigma > \sigma,
\]
and $|z-y|<\sigma$,
we have 
\begin{equation*}
K(z,y,\sigma^\theta)\ge C\sigma^{-N}d(y)^{-\frac{\theta}{2}}.
\end{equation*}
Furthermore, since 
\[
d(y) < d(z) + \sigma \le  d(z) + T^\frac{1}{\theta} \le 2d(z),
\]
we obtain 
\[
K(z,y,\sigma^\theta)\ge { C \sigma^{-N}d(z)^{-\frac{\theta}{2}}}.
\]
Thus, Lemma \ref{Lemma:3.1} follows.
\end{proof}

\begin{lemma}
\label{Lemma:3.2}
\begin{itemize}
\item[(i)] One has
\[
\Gamma_\theta(x,2t-s) \ge \left(\frac{s}{2t}\right)^\frac{N}{\theta} \Gamma_\theta (x,s)
\]
for all $x\in \mathbb{R}^N$ and $s,t>0$ with $s<t$.
\item[(ii)] There exists $C>0$ such that 
\begin{equation}
\label{eq:3.3}
G(z,y,2t-s) \ge C \left(\frac{s}{2t}\right)^\frac{N}{\theta} G(z,y,s)
\end{equation}
 for all $z\in \Omega$ with $d(z) \ge  T^{1/\theta}$, $y\in \Omega$, $s,t \in (0,T/32)$ with $s<t$, and $T\in (0,T_*]$.  
 \end{itemize}
\end{lemma}

\begin{proof}
Assertion (i) has been already proved in \cite{HI01}.
We prove assertion (ii).
Let  $z\in \Omega$ with $d(z) \ge T^{1/\theta}$, $y\in \Omega$, $s,t \in (0,T/32)$ with $s<t$, and $T\in(0,T_*]$.
By \eqref{eq:1.2} we have
\begin{equation}
\label{eq:3.4}
\begin{split}
&G(z,y,2t-s) \\
&\ge C\left(1\wedge \frac{d(z)^\frac{\theta}{2}}{\sqrt{2t-s}}\right)\left(1\wedge \frac{d(y)^\frac{\theta}{2}}{\sqrt{2t-s}}\right)\Gamma_\theta(z-y,2t-s).
\end{split}
\end{equation} 
Since $d(z) \ge T^{{1/\theta}} >(2t-s)^{{1/\theta}}> s^{{1/\theta}}$, we see that
\[
\left(1\wedge \frac{d(z)^\frac{\theta}{2}}{\sqrt{2t-s}}\right) = \left(1\wedge \frac{d(z)^\frac{\theta}{2}}{\sqrt{s}}\right) =1.
\]
We assume that $d(y)\ge (2t-s)^{1/\theta} (>s^{1/\theta})$. Similarly, we see that
\[
\left(1\wedge \frac{d(y)^\frac{\theta}{2}}{\sqrt{2t-s}}\right) = \left(1\wedge \frac{d(y)^\frac{\theta}{2}}{\sqrt{s}}\right) =1.
\]
By \eqref{eq:3.4} and assertion (i) we then  have 
\begin{equation*}
\begin{split}
&G(z,y,2t-s)\\
&\ge C\left(1\wedge \frac{d(z)^\frac{\theta}{2}}{\sqrt{s}}\right)\left(1\wedge \frac{d(y)^\frac{\theta}{2}}{\sqrt{s}}\right)\Gamma_\theta(z-y,2t-s)\\
& = C\left(\frac{s}{2t}\right)^\frac{N}{\theta}\left(1\wedge \frac{d(z)^\frac{\theta}{2}}{\sqrt{s}}\right)\left(1\wedge \frac{d(y)^\frac{\theta}{2}}{\sqrt{s}}\right)\Gamma_\theta(z-y,s)\\
&\ge C \left(\frac{s}{2t}\right)^\frac{N}{\theta}G(z,y,s).
\end{split}
\end{equation*}
We obtained the desired inequality.

On the other hand, we assume that  $d(y)\le (2t-s)^{1/\theta}$.
Note that
\begin{equation*}
\begin{split}
|z-y|^{N+\theta} 
&> (d(z) - d(y))^{N+\theta} \\
&> (T^\frac{1}{\theta} - (2t-s)^\frac{1}{\theta})^{N+\theta} \\
&> (2t-s)^{\frac{N}{\theta}+1}.
\end{split}
\end{equation*}
By  \eqref{eq:2.1} and \eqref{eq:3.4} we then  have 
\begin{equation*}
\begin{split}
G(z,y,2t-s)
&\ge C\left(1\wedge \frac{d(z)^\frac{\theta}{2}}{\sqrt{s}}\right) \frac{d(y)^\frac{\theta}{2}}{\sqrt{2t-s}}
\Gamma_\theta(z-y,2t-s)\\
&= C\left(1\wedge \frac{d(z)^\frac{\theta}{2}}{\sqrt{s}}\right) \frac{d(y)^\frac{\theta}{2}}{\sqrt{2t-s}}\frac{2t-s}{|z-y|^{N+\theta}}\\
&= C\sqrt{\frac{2t-s}{s}} \left(1\wedge \frac{d(z)^\frac{\theta}{2}}{\sqrt{s}}\right) \frac{d(y)^\frac{\theta}{2}}{\sqrt{s}}\frac{s}{|z-y|^{N+\theta}}\\
&\ge C\left(1\wedge \frac{d(z)^\frac{\theta}{2}}{\sqrt{s}}\right) \left(1\wedge \frac{d(y)^\frac{\theta}{2}}{\sqrt{s}}\right)\left(s^{-\frac{N}{\theta}}\wedge\frac{s}{|z-y|^{N+\theta}}\right)\\
&\ge CG(z,y,s)\ge C \left(\frac{s}{2t}\right)^\frac{N}{\theta}G(z,y,s).
\end{split}
\end{equation*}
We obtained the desired inequality and the proof is complete.
\end{proof}

\begin{proof}[Proof of Proposition \ref{Proposition:3.1}.]
Let $u$ be a supersolution of problem {\rm (SHE)} in $Q_T$, where $T\in(0,T_*]$.
Let $\sigma\in (0,T^{1/\theta}/16)$ and $z\in \Omega$ with $d(z) \ge T^{1/\theta}$.
It follows from \eqref{eq:0.3} and  Lemmas \ref{Lemma:2.2} and \ref{Lemma:3.2} that
\begin{equation*}
\begin{split}
&\int_\Omega G(z,x,t) u(x,t) \,ds\\
&\quad\ge \int_{\overline{\Omega}} \int_\Omega G(z,x,t)K(x,y,t) \, dxd\mu(y)\\
& \qquad\qquad + \int_0^t \int_{\Omega} \int_\Omega G(z,x,t)G(x,y,t-s) u(y,s)^p \,dxdyds\\
& \quad\ge \int_{\overline{\Omega}} K(z,y,2t) \, d\mu(y)  + \int_{\sigma^\theta}^t G(z,y,2t-s)u(y,s)^p \, dyds\\
&\quad \ge \int_{\overline{\Omega}} K(z,y,2t) \,d\mu(y) + C\int_{\sigma^\theta}^t\left(\frac{s}{2t}\right)^\frac{N}{\theta} \int_{\Omega}G(z,y,s)u(y,s)^p\,dyds
\end{split}
\end{equation*}
for almost all $t\in (\sigma^\theta,T/32)$.
Furthermore, Jensen's inequality with \eqref{eq:0.2} implies that
\[
\int_\Omega G(z,y,s)u(y,s)^p\,dy \ge \left(\int_\Omega G(z,y,s)u(y,s)\,dy \right)^p 
\]
for almost all $s>0$.
Then we obtain 
\begin{equation}
\label{eq:3.5}
\begin{split}
&\int_\Omega G(z,x,t) u(x,t) \,ds\\
& \ge \int_{\overline{\Omega}} K(z,y,2t) \,d\mu(y)  +Ct^{-\frac{N}{\theta}}\int_{\sigma^\theta}^t s^{\frac{N}{\theta}}\left(\int_\Omega G(z,y,s)u(y,s)\,dy \right)^p \,ds
\end{split}
\end{equation}
for almost all $t\in (\sigma^\theta,T/32)$.
In addition, Lemma \ref{Lemma:3.1} implies that
\begin{equation}
\label{eq:3.6}
\begin{split}
\int_{\overline{\Omega}} K(z,y,2t) \, d\mu(y) 
&\ge \int_{B_\Omega (z,(2t)^\frac{1}{\theta})} K(z,y,2t) \, d\mu(y)\\
& \ge C t^{-\frac{N}{\theta}} d(z)^{-\frac{\theta}{2}} \mu (B_\Omega(z,(2t)^\frac{1}{\theta}))\\
& \ge C t^{-\frac{N}{\theta}} d(z)^{-\frac{\theta}{2}} \mu (B_\Omega(z,\sigma))\\
\end{split}
\end{equation}
for all  $t\in (\sigma^\theta,T/32)$. Therefore, setting 
\[
U(t) := t^\frac{N}{\theta} \int_{\Omega} G(z,y,t) u(y,t) \,dy,
\]
by \eqref{eq:3.5} and \eqref{eq:3.6} we obtain
\[
U(t) \ge C  d(z)^{-\frac{\theta}{2}} \mu (B_\Omega(z,\sigma)) + C\int_{\sigma^\theta}^t s^{-\frac{N}{\theta}(p-1)} U(s)^p\,ds
\]
for almost all $t\in (\sigma^\theta,T/32)$.
Applying Lemma \ref{Lemma:2.4}, we obtain 
\[
d(z)^{-\frac{\theta}{2}} \mu (B_\Omega(z,\sigma)) \le C (\sigma^\theta)^{\frac{N}{\theta}-\frac{1}{p-1}} = C \sigma^{N-\frac{\theta}{p-1}}
\]
for all $\sigma\in (0, T^{1/\theta}/16)$ and almost all $z\in \Omega$ with $d(z)\ge T^{1/\theta}$, so that \eqref{eq:3.1} holds for all  $\sigma\in (0, T^{1/\theta}/16)$ and  $z\in \Omega$ with $d(z)\ge T^{1/\theta}$.
Furthermore, in the case of $p=p_\theta(N,0)$, we have 
\[
d(z)^{-\frac{\theta}{2}} \mu (B_\Omega(z,\sigma)) \le C \left[\log\frac{T}{2\sigma^\theta}\right]^{-\frac{N}{\theta}} \le C \left[\log\left(e+\frac{T^\frac{1}{\theta}}{\sigma}\right)\right]^{-\frac{N}{\theta}}
\]
for all $\sigma\in (0, T^{1/\theta}/16)$ and almost all $z\in \Omega$ with $d(z)\ge T^{1/\theta}$, so that \eqref{eq:3.2} holds for all  $\sigma\in (0, T^{1/\theta}/16)$ and  $z\in \Omega$ with $d(z)\ge T^{1/\theta}$.
Thus, Proposition \ref{Proposition:3.1} holds.
\end{proof}

Next we prove Proposition \ref{Proposition:3.2} on the behavior of $\mu$ near the boundary.

\begin{proposition}
\label{Proposition:3.2}
Assume that there exists a supersolution of problem {\rm (SHE)} in $Q_T$, where $T\in (0,T_*]$.
Then there exist $\gamma>0$ and $\epsilon\in(0,1)$ such that
\[
\mu(B_\Omega(z,\sigma)) \le \gamma \sigma^{N+\frac{\theta}{2}-\frac{\theta}{p-1}} 
\]
for all $z\in \partial\Omega$ and $\sigma\in (0,\epsilon T^\frac{1}{\theta})$.
\end{proposition}

In order to prove Proposition \ref{Proposition:3.2}, we prepare Lemma \ref{Lemma:3.3}.

\begin{lemma}
\label{Lemma:3.3}
Let $u$ be a solution of problem {\rm (SHE)} in $Q_T$, where $T\in (0,T_*]$. Then there exists $C>0$ such that
\[
u(x,(2\sigma)^\theta) \ge C \sigma^{-N-\frac{\theta}{2}} \mu(B_\Omega(z,\sigma))
\]
for all $z\in \partial \Omega$, almost all $x\in B_\Omega(z, 8\sigma)$ with $d(x)\in (2\sigma,4\sigma)$, and almost all $\sigma\in (0, T^{1/\theta}/16)$.
\end{lemma}
\begin{proof}
Let $z\in\partial \Omega$.
For any $x\in B_\Omega(z,8\sigma)$ with $d(x)\in (2\sigma,4\sigma)$ and $y\in B_\Omega(z,\sigma)$,
by \eqref{eq:1.3} and \eqref{eq:2.1} we have
\begin{equation}
\label{eq:3.7}
\begin{split}
&K(x,y,(2\sigma)^\theta)\\
&\ge C \left(1\wedge \frac{d(x)^\frac{\theta}{2}}{(2\sigma)^\frac{\theta}{2}}\right)
\left(\frac{1}{d(y)^\frac{\theta}{2}}\wedge \frac{1}{(2\sigma)^\frac{\theta}{2}}\right)
\left((2\sigma)^{-N} \wedge \frac{(2\sigma)^\theta}{|x-y|^{N+\theta}}\right).
\end{split}
\end{equation}
Since 
\[
d(x)>2\sigma, \quad d(y)<\sigma, \quad \mbox{and} \quad |x-y| \le |x-z|+|z-y| \le 9\sigma,
\]
\eqref{eq:3.7} implies that
\[
K(x,y,(2\sigma)^\theta) \ge C \sigma^{-N-\frac{\theta}{2}}. 
\]
Then it follows from Definition \ref{Def:1.1} that
\[
u(x,(2\sigma)^\theta) \ge \int_{B_\Omega(z,\sigma)} K(x,y,(2\sigma)^\theta) \, d\mu(y) \ge C \sigma^{-N-\frac{\theta}{2}} \mu (B_\Omega(z,\sigma))
\]
for all $z\in \partial \Omega$, almost all $x\in B_\Omega(z, 8\sigma)$ with $d(x)\in (2\sigma,4\sigma)$, and almost all $\sigma\in (0, T^{1/\theta}/16)$.
Thus, Lemma \ref{Lemma:3.3} follows.
\end{proof}

\begin{proof}[Proof of Proposition \ref{Proposition:3.2}]
Assume that there exists a supersolution of problem {\rm (SHE)} in $Q_T$, where $T\in (0,T_*]$. 
Let $\epsilon \in (0,1/16)$. For $\sigma \in (0,\epsilon T^{1/\theta})$, we have 
\[
T-(2\sigma)^\theta> (1-4\epsilon^\theta) T > \frac{T}{2}.
\]
Set $\tilde{u}(x,t) := u(x, t+(2\sigma)^\theta)$. Then,  for almost all $\sigma \in (0,\epsilon T^{1/\theta})$, the function $\tilde{u}$ is a supersolution of problem {\rm (SHE)} with $\mu = d(x)^{\theta/2} u(x, (2\sigma)^\theta)$ in $Q_{T/2}$.
For $z\in \partial \Omega$, let $\tilde{z}\in \Omega$ be such that $\tilde{z}\in \partial B_\Omega(z,3\sigma)$ and $d(\tilde{z}) = 3\sigma$.
Let $\delta\in (0,3/16)$. 
Since $\epsilon T^{1/\theta} < T^{1/\theta}/16$  and $y\in B_\Omega(\tilde{z},\delta\sigma)$ satisfies $y\in B_\Omega(z,8\sigma)$ and $d(y)\in (2\sigma,4\sigma)$, by Lemma \ref{Lemma:3.3} we have
\begin{equation}
\label{eq:3.8}
\begin{split}
&\int_{B_\Omega(\tilde{z},\delta \sigma)} d(y)^\frac{\theta}{2} u(y,(2\sigma)^\theta) \, dy\\
& \ge C \sigma^{-N-\frac{\theta}{2}} \mu (B_\Omega(z,\sigma))\int_{B_\Omega(\tilde{z},\delta\sigma)} d(y)^\frac{\theta}{2}\, dy \ge C \mu (B_\Omega(z,\sigma)).
\end{split}
\end{equation} 
On the other hand, applying Proposition \ref{Proposition:3.1} with $T= (3\sigma)^\theta$ to $\tilde{u}$, we have 
\begin{equation*}
\begin{split}
d(\tilde{z})^{-\frac{\theta}{2}} \int_{B_\Omega(\tilde{z},\delta \sigma)} d(y)^\frac{\theta}{2} u(y,(2\sigma)^\theta) \, dy 
&=
d(\tilde{z})^{-\frac{\theta}{2}} \int_{B_\Omega(\tilde{z},\delta \sigma)} d(y)^\frac{\theta}{2} \tilde{u}(y,0) \, dy\\
&\le C \sigma^{N-\frac{\theta}{p-1}}.
\end{split}
\end{equation*}
This together with \eqref{eq:3.8} implies that 
\[
\mu(B_\Omega(z,\sigma)) \le C\sigma^{N+\frac{\theta}{2}-\frac{\theta}{p-1}}
\]
for all $z\in\partial\Omega$ and almost all $\sigma\in (0,\epsilon T^{1/\theta})$. Then we obtain the desired inequality for all $z\in \partial \Omega$ and all $\sigma \in (0,\epsilon T^{1/\theta})$. Thus, Proposition \ref{Proposition:3.2} follows.
\end{proof}

Now we are ready to complete the proof of Theorem \ref{Theorem:3.1}.

%
\begin{proof}[Proof of \eqref{Thm:3.1.1} and \eqref{Thm:3.1.2}.]
By Propositions \ref{Proposition:3.1} and \ref{Proposition:3.2} we find $\delta\in (0,1/3)$ such that
\begin{equation}
\label{eq:3.8.1}
\begin{split}
&\sup_{z\in\Omega, d(z)\ge \sigma} d(z)^{-\frac{\theta}{2}} \mu(B_\Omega(z,\delta\sigma)) \le C \sigma^{N-\frac{\theta}{p-1}},\\
&\sup_{z\in\partial \Omega} \mu(B_\Omega(z,\delta\sigma)) \le C \sigma^{N+\frac{\theta}{2}-\frac{\theta}{p-1}},
\end{split}
\end{equation}
for all $\sigma\in (0,T^\frac{1}{\theta})$.
Furthermore, if $p=p_\theta(N,0)$, then 
\begin{equation}
\label{eq:3.8.2}
\sup_{z\in\Omega, d(z)\ge \sigma} d(z)^{-\frac{\theta}{2}} \mu (B_\Omega (z,\delta\sigma)) \le C \left[\log\left(e+\frac{T^\frac{1}{\theta}}{\sigma}\right)\right]^{-\frac{N}{\theta}}
\end{equation}
for all $\sigma\in (0,T^{1/\theta})$.

Let $\sigma\in (0,T^{1/\theta})$ and $z\in\overline{\Omega}$.
Consider the case of $0\le d(z) \le \delta\sigma/2$. Since $0<\delta<1/3$, we have
\[
B_\Omega (z,\delta\sigma/2) \subset B_\Omega (\zeta, \delta \sigma) \subset B_\Omega (z,\sigma),
\]
where $\zeta\in \overline{B_\Omega(z,\delta\sigma/2)}\cap \partial\Omega\neq \emptyset$.
Then by \eqref{eq:3.8.1} we obtain
\begin{equation}
\label{eq:3.8.3}
\begin{split}
\mu(B_\Omega (z,\delta\sigma/2))
& \le \mu(B_\Omega (\zeta,\delta\sigma)) \le C\sigma^{N+\frac{\theta}{2}-\frac{\theta}{p-1}}\\
&\le C\sigma^{-\frac{\theta}{p-1}} \int_{B_\Omega(\zeta,\delta\sigma)\cap \{y\in\Omega;d(y)\ge \delta\sigma/3\}} d(y)^\frac{\theta}{2} \,dy\\
&\le C\sigma^{-\frac{\theta}{p-1}} \int_{B_\Omega(\zeta,\delta\sigma)} d(y)^\frac{\theta}{2} \,dy\\
&\le C \sigma^{-\frac{\theta}{p-1}} \int_{B_\Omega (z,\sigma)} d(y)^\frac{\theta}{2}\,dy. 
\end{split}
\end{equation}

Consider the case of $d(z)>\delta\sigma/2$. Then, by \eqref{eq:3.8.1} we have 
\begin{equation}
\label{eq:3.8.4}
\begin{split}
\mu(B_\Omega(z,\delta^2\sigma)) 
&\le Cd(z)^\frac{\theta}{2} \sigma^{N-\frac{\theta}{p-1}} 
\le Cd(z)^\frac{\theta}{2} \sigma^{-\frac{\theta}{p-1}} \int_{B_\Omega (z, \delta^2\sigma/4)}\,dy \\
&\le C\sigma^{-\frac{\theta}{p-1}} \int_{B_\Omega(\delta^2\sigma/4)} d(y)^\frac{\theta}{2}\,dy \\
&\le C \sigma^{-\frac{\theta}{p-1}} \int_{B_\Omega(z,\sigma)} d(y)^\frac{\theta}{2} \,dy.
\end{split}
\end{equation}
Combining \eqref{eq:3.8.3} and \eqref{eq:3.8.4}, we obtain
\begin{equation}
\label{eq:3.8.5}
\mu(B_\Omega(z,\delta^2\sigma/2)) \le C\sigma^{-\frac{\theta}{p-1}} \int_{B_\Omega(z,\sigma)} d(y)^\frac{\theta}{2} \,dy
\end{equation}
for $z\in\overline{\Omega}$ and $\sigma\in(0,T^{1/\theta})$.
Therefore, by Lemma \ref{Lemma:2.0} (ii) and \eqref{eq:3.8.5}, for any $z\in\overline{\Omega}$,
we find $\{\overline{z}_i\}_{i=1}^m\subset B_\Omega(z,2\sigma)$ such that
\begin{equation*}
\begin{split}
\mu(B_\Omega(z,\sigma)) 
&\le \sum_{i=1}^m\mu(B_\Omega(\overline{z}_i,\delta^2\sigma/2))\\
&\le C\sigma^{-\frac{\theta}{p-1}} \sum_{i=1}^m \int_{B_\Omega(\overline{z}_i,\sigma)} d(y)^\frac{\theta}{2}\,dy\\
&\le C\sigma^{-\frac{\theta}{p-1}} \int_{B_\Omega(z,3\sigma)} d(y)^\frac{\theta}{2}\,dy\\
&\le C\sigma^{-\frac{\theta}{p-1}} \int_{B_\Omega(z,\sigma)} d(y)^\frac{\theta}{2}\,dy.\\
\end{split}
\end{equation*}
This implies assertion (i).

Similarly, if $p=p_\theta(N,0)$, then, by Lemma \ref{Lemma:2.0}, for any $z\in \Omega$ with $d(z)\ge 3\sigma$,
we find $\{\tilde{z}_i\}_{i=1}^{m'}\subset B_\Omega (z,2\sigma)$ such that
\[
\mu(B_\Omega(z,\sigma)) \le \sum_{i=1}^{m'} \mu(B_\Omega(\tilde{z}_i,\delta\sigma)). 
\]
Since $\tilde{z}_i$ satisfies $d(\tilde{z}_i)\ge \sigma$ and $0<\delta<1/3$, we deduce from \eqref{eq:3.8.2} that
\begin{equation*}
\begin{split}
d(z)^{-\frac{\theta}{2}} \mu(B_\Omega(z,\sigma)) 
&\le C\sum_{i=1}^{m'} \left(\frac{d(z)+2\sigma}{d(z)}\right)^\frac{\theta}{2}  \left[\log\left(e+\frac{T^\frac{1}{\theta}}{\sigma}\right)\right]^{-\frac{N}{\theta}} \\
&\le  C\left[\log\left(e+\frac{T^\frac{1}{\theta}}{\sigma}\right)\right]^{-\frac{N}{\theta}} \\
\end{split}
\end{equation*}
for all $z\in\Omega$ with $d(z) \ge 3\sigma$ and $\sigma\in (0,T^{1/\theta})$.
This implies assertion (ii), and the proof is complete.
\end{proof}

In the case of $p=p_\theta(N,\theta/2)$, we obtain more delicate estimates of $\mu$ near the boundary than those of \eqref{Thm:3.1.1}.

\begin{proof}[Proof of \eqref{Thm:3.1.3}.]
Let $p=p_\theta(N,\theta/2)$. Assume that there exists a supersolution of problem {\rm (SHE)} in $Q_T$, where $T\in (0,T_*]$. 

Let $z\in\partial \Omega$. By Lemma \ref{Lemma:2.4}, for almost all $\sigma\in (0,T^{1/\theta}/3)$, the function $v(x,t):= u(x,t+(2\sigma)^\theta)$ is a solution of problem {\rm (SHE)} in $Q_{T-(2\sigma)^\theta}$. It follows from \eqref{Thm:3.1.1} that
\begin{equation}
\label{eq:3.9}
\int_{B_\Omega(z,r)} d(y)^\frac{\theta}{2} v(y,t) \, dy \le C r^{-\frac{\theta}{p-1}} \int_{B_\Omega(z,r)} d(y)^\frac{\theta}{2} \,dy
\end{equation}
for all $r\in (0, (T-(2\sigma^\theta)-t))^{1/\theta}$ and almost all $t\in (0,T-(2\sigma)^\theta)$. Then
\[
V(t) := t^{\frac{N}{\theta}+1} \int_\Omega K(x,z,t)v(x,t) \, dx <\infty 
\]
for almost all $ t\in(\sigma^\theta,(T-(2\sigma)^\theta)/2)$.
Indeed, by Lemma \ref{Lemma:2.1} and \eqref{eq:3.9} we have 
\begin{equation*}
\begin{split}
\int_\Omega K(x,z,t) v(x,t) \, dx 
&\le  Ct^{-1} \int_\Omega \Gamma_\theta (x-z,t) d(x)^\frac{\theta}{2} v(x,t) \,dx\\
&\le Ct^{-\frac{N}{\theta}-1} \sup_{z\in \overline{\Omega}} \int_{B_\Omega (z, (2t)^\frac{1}{\theta})} d(y)^\frac{\theta}{2} v(y,t)\,dy <\infty
\end{split}
\end{equation*}
for almost all $t\in (\sigma^\theta,(T-(2\sigma)^\theta)/2)$.

We derive an integral inequality for $V$. By Fubini's theorem and \eqref{eq:2.8} 
we have
\begin{equation}
\label{eq:3.10}
\begin{split}
&\int_\Omega K(x,z,t) v(x,t) \, dx\\
&\ge \int_\Omega \int_\Omega K(x,z,t)G(x,y,t) v(y,0) \,dxdy\\
&\qquad\qquad+ \int_0^t \int_\Omega \int_\Omega K(x,z,t) G(x,y,t-s) v(y,s)^p \,dxdyds\\
&\ge \int_{\Omega} K(y,z,2t) v(y,0)\,dy+ \int_0^t  \int_\Omega K(y,z,2t-s)  v(y,s)^p \,dyds.\\
\end{split}
\end{equation}
Set
\[
I:= \{y\in B_\Omega (z, 8\sigma); d(y)\in (2\sigma, 4\sigma)\}.
\]
For $y\in I$, by \eqref{eq:1.3} and \eqref{eq:2.1} we have 
\begin{equation}
\label{eq:3.11}
\begin{split}
&K(y,z,2t)\ge C \left(1\wedge \frac{d(y)^\frac{\theta}{2}}{\sqrt{2t}}\right)\frac{1}{\sqrt{2t}} \left((2t)^{-\frac{N}{\theta}} \wedge \frac{2t}{|y-z|^{N+\theta}}\right)
\end{split}
\end{equation}
Since 
\[
d(y)^\frac{\theta}{2} < 4^\frac{\theta}2{}\sigma^\frac{\theta}{2} < 4^\frac{\theta}{2}\sqrt{t} \quad \mbox{and} \quad
|y-z| < 8\sigma <8t^\frac{1}{\theta}
\]
for $y\in I$ and  $ t\in(\sigma^\theta,(T-(2\sigma)^\theta)/2)$, \eqref{eq:3.11} implies that
\[
K(y,z,2t) \ge Cd(y)^\frac{\theta}{2} t^{-\frac{N}{\theta}-1}
\]
for $y\in I$ and  $ t\in(\sigma^\theta,(T-(2\sigma)^\theta)/2)$.
Then we have
\begin{equation}
\label{eq:3.12}
\int_{\Omega} K(y,z,2t) v(y,0)\,dy \ge C t^{-\frac{N}{\theta}-1} \int_I d(y)^\frac{\theta}{2} v(y,0)\,dy
\end{equation}
for all $ t\in(\sigma^\theta,(T-(2\sigma)^\theta)/2)$.
On the other hand, by the same argument as in the proof of \eqref{eq:3.3}, we obtain 
\[
K(y,z,2t-s) \ge C\left(\frac{s}{2t}\right)^{-\frac{N}{\theta}+1} K(y,z,s)
\]
for all $y\in \Omega$, $z\in \partial \Omega$, and $s,t\in (0,T)$ with $s<t$.
Then Jensen's inequality  with \eqref{eq:2.6} implies that
\begin{equation}
\label{eq:3.13}
\begin{split}
 &\int_0^t  \int_\Omega K(y,z,2t-s)  v(y,s)^p \,dyds\\
 &\ge \int_0^t \left(\frac{s}{2t}\right)^{\frac{N}{\theta}+1}  C_4s^{-\frac{1}{2}} \int_\Omega C_4^{-1} s^\frac{1}{2}K(y,z,s)v(y,s)^p\,dyds\\
 &\ge C\int_0^t \left(\frac{s}{2t}\right)^{\frac{N}{\theta}+1}  s^{-\frac{1}{2}} \left(\int_\Omega  s^\frac{1}{2}K(y,z,s)v(y,s)\,dy\right)^p\,ds\\
  &\ge C t^{-\frac{N}{\theta}-1}\int_{\sigma^\theta}^t   s^{-\left(\frac{N}{\theta}+\frac{1}{2}\right)(p-1)} \left(\int_\Omega  s^{\frac{N}{\theta}+1}K(y,z,s)v(y,s)\,dy\right)^p\,ds.\\
\end{split}
\end{equation}
Since $p=p_\theta(N,\theta/2)$, by \eqref{eq:3.10}, \eqref{eq:3.12}, and \eqref{eq:3.13} we see that
\begin{equation}
\label{eq:3.14}
\begin{split}
V(t) \ge C \int_I d(y)^\frac{\theta}{2} u(y,(2\sigma)^\theta) \, dy + C \int_{\sigma^\theta}^t s^{-1} V(s)^p\, ds
\end{split}
\end{equation}
for almost all $t\in (\sigma^\theta, (T-(2\sigma)^\theta)/3)$ and almost all $\sigma\in (0, T^{1/\theta}/3)$.

Let $\epsilon \in (0,1/2)$. We apply Lemma \ref{Lemma:2.4} to inequality \eqref{eq:3.14}. Then
\begin{equation}
\label{eq:3.15}
\begin{split}
 \int_I d(y)^\frac{\theta}{2} u(y,(2\sigma)^\theta) \, dy 
 &\le C\left[\log\frac{T}{\sigma^\theta}\right]^{-\frac{2N+\theta}{2\theta}} \\
 &\le C \left[\log\left(e+\frac{T^\frac{1}{\theta}}{\sigma}\right)\right]^{-\frac{2N+\theta}{2\theta}} 
 \end{split}
\end{equation}
for almost all $\sigma\in (0,\epsilon T^{1/\theta})$
Therefore, by Lemma \ref{Lemma:3.3}, taking small enough $\epsilon >0$ if necessary, we have
\begin{equation}
\label{eq:3.16}
\begin{split}
&\int_I d(y)^\frac{\theta}{2} u(y,(2\sigma)^\theta) \, dy\\
& \ge C \sigma^{-N-\frac{\theta}{2}} \mu(B_\Omega(z,\sigma)) \int_{I} d(y)^\frac{\theta}{2} \, dy \ge C \mu(B_\Omega(z,\sigma)) 
\end{split}
\end{equation}
for almost all $\sigma\in(0,\epsilon T^{1/\theta})$.

Combining \eqref{eq:3.15} and \eqref{eq:3.16}, we find $\delta\in (0,1)$ such that 
\[
\sup_{z\in\partial \Omega} \mu(B_\Omega(z,\delta\sigma))  \le C \left[\log\left(e+\frac{T^\frac{1}{\theta}}{\sigma}\right)\right]^{-\frac{2N+\theta}{2\theta}} 
\]
for almost all $\sigma\in (0, T^{1/\theta})$.
This together with Lemma \ref{Lemma:2.0} implies that 
\[
\sup_{z\in \partial \Omega} \mu(B_\Omega (z,\sigma)) \le \sum_{i=1}^{m'} \mu (B_\Omega (z_i',\delta \sigma )) \le C \left[\log\left(e+\frac{T^\frac{1}{\theta}}{\sigma}\right)\right]^{-\frac{2N+\theta}{2\theta}} 
\]
for all $\sigma \in (0, T^{1/\theta})$. Thus, \eqref{Thm:3.1.3} follows.
\end{proof}
\section{Sufficient conditions for the local-in-time solvability.}
In this section we study sufficient conditions on the solvability of problem {\rm (SHE)}.
Denote ${\mathcal L}$ and ${\mathcal L}'$ by the set of nonnegative measurable functions on $\Omega$ and  the set of nonnegative measurable functions on $\partial\Omega$, respectively.
For $\mu\in \mathcal{M}$ and $h\in {\mathcal L}'$, define
\begin{equation*}
\begin{split}
&[\G(t)\mu](x) := \int_{\Omega} G(x,y,t)\,d\mu(y),\\
&[\K(t)h](x) := \frac{C_5t^{\frac{1}{2}+\frac{1}{\theta}}}{D(x,t)} \int_{\partial \Omega}K(x,y,t)h(y)\,d\sigma(y),
\end{split}
\end{equation*}
for $x\in \overline{\Omega}$, where $C_5$ is the constant as in \eqref{eq:2.7}.

We first show that the existence of solutions and supersolutions of problem {\rm (SHE)} are equivalent.
The arguments in the proofs of sufficient conditions are based on Lemma \ref{Lemma:2.3}.
\begin{lemma}
\label{Lemma:2.3}
Assume that there exists a supersolution $v$ of problem {\rm (SHE)} in $Q_T$.
Then problem {\rm (SHE)} possesses a solution $u$ in $Q_T$ such that $u\le v$ in $Q_T$.
\end{lemma}
\begin{proof}
This lemma can be proved by the same argument as in \cite[Lemma 2.2]{HI01}.
Define 
\begin{align*}
& u_1(x,t) := \int_{\overline{\Omega}}K(x,y,t) \, d\mu (y),\\
& u_{j+1} := u_1(x,t) + \int_0^t\int_{\Omega} G(x,y,t-s)u_j(y,s)^p\,dyds, \quad j=1,2,\cdots,
\end{align*}
for almost all $(x,t)\in Q_T$.
Thanks to \eqref{Def:1.1.2} and the nonnegativity of $K$ and $G$, by induction we obtain
\[
0\le u_1(x,t)\le u_2(x,t)\le\cdots\le u_j(x,t)\le \cdots \le v(x,t)<\infty
\]
for almost all $(x,t)\in Q_T$. Then the limit function
\[
u(x,t) := \lim_{j\to\infty}u_j(x,t)
\]
is well-defined for almost all $(x,t)\in Q_T$ and it is a solution of problem {\rm (SHE)} in $Q_T$
such that
$u(x,t)\le v(x,t)$ for almost all $(x,t)\in Q_T$. Then the proof is complete.
\end{proof}

\subsection{The case of $\mu\in \mathcal{M}$.}
We begin with the case of $\mu \in \mathcal{M}$.
\begin{theorem}
\label{Theorem:4.1}
Let $N\ge 1$, $ p>1$, and $0<\theta<2$. Then there exists $\gamma=\gamma(\Omega, N, p, \theta)>0$ such that,
if $\mu\in\mathcal{M}$ satisfies
\begin{equation}
\label{eq:4.1}
\int_0^Ts^{-\frac{N}{\theta}(p-1)} \left(\sup_{z\in \overline{\Omega}} \int_{B_\Omega(z,s^\frac{1}{\theta})}\frac{d\mu(y)}{\dist(y)^\frac{\theta}{2}+\sqrt{s}}\right)^{p-1}\,ds \le \gamma
\end{equation}
for some $T\in (0,T_*]$, then problem {\rm (SHE)} possesses a solution in $Q_T$.
\end{theorem}

\begin{proof}
Assume \eqref{eq:4.1}. Let $T\in(0,T_*]$ and 
\[
\displaystyle{w(x,t) := 2 \int_{\overline{\Omega}} K(x,y,t)\,d\mu(y)}.
\]
It follows from \eqref{eq:2.4} that
\[
\|w(t)\|_{L^\infty(\Omega)} \le C t^{-\frac{N}{\theta}} \sup_{z\in\overline{\Omega}} \int_{B_\Omega(z,t^\frac{1}{\theta})}\frac{d\mu(y)}{\dist(y)^\frac{\theta}{2}+\sqrt{t}}
\]
for all $t>0$.
Then, by \eqref{eq:2.8} we have
\begin{equation*}
\begin{split}
&\int_{\overline{\Omega}} K(x,y,t) \, d\mu(y) + \int_0^t \int_{\Omega} G(x,y,t-s) w(y,s)^p\,dyds\\
&\le \frac{1}{2}w(x,t) + \int_0^t \|w(s)\|_{L^\infty(\Omega)}^{p-1} \int_{\Omega} G(x,y,t-s) w(y,s)\,dyds\\
&\le \frac{1}{2}w(x,t) + C\int_0^t \|w(s)\|_{L^\infty(\Omega)}^{p-1} \int_{\overline{\Omega}}\int_{\Omega} G(x,y,t-s) K(y,z,s)\,dyd\mu(z)ds\\
&= \frac{1}{2}w(x,t) + C\int_0^t \|w(s)\|_{L^\infty(\Omega)}^{p-1} \,ds\int_{\overline{\Omega}} K(x,z,t)\,d\mu(z)\\
&\le \frac{1}{2}w(x,t)\\
& + C\int_0^T s^{-\frac{N}{\theta}(p-1)} \left(\sup_{z\in\overline{\Omega}} \int_{B_\Omega(z,s^\frac{1}{\theta})}\frac{d\mu(y)}{\dist(y)^\frac{\theta}{2}+\sqrt{s}}\right)^{p-1} \,ds\int_{\overline{\Omega}} K(x,z,t)\,d\mu(z)\\
&\le  \frac{1}{2}w(x,t) + C\gamma w(x,t)
\end{split}
\end{equation*}
for almost all $(x,t) \in Q_T$. Taking sufficiently small $\gamma>0$ if necessary, we see that $w$ is a supersolution of {\rm (SHE)}. Thus, Theorem \ref{Theorem:4.1} follows from Lemma \ref{Lemma:2.3}.
\end{proof}

%
%
\begin{corollary}
Let $N\ge1$ and $\delta_N$ be the $N$-dimensional Dirac measure concentrated at the origin.
Let $\kappa>0$ and $z\in \partial \Omega$.
If $\mu = \kappa \delta_N(\cdot-z)$ on $\overline{\Omega}$, 
then the following holds:
\begin{itemize}
\item[(i)] If $p\ge p_\theta(N,\theta/2)$, the problem {\rm (SHE)} possesses no local-in-time solution; 
\item[(ii)] If $1<p<p_\theta(N,\theta/2)$, then problem {\rm (SHE)}  possesses a local-in-time solution. 
\end{itemize}
\end{corollary}

\subsection{More delicate cases.}
%
In this subsection we modify the arguments in \cite{FHIL, HI01, HIT2, RS} to obtain Theorem \ref{Theorem:4.2} on sufficient conditions on the solvability of problem (SHE).

\begin{theorem}
\label{Theorem:4.2}
Let $f\in \mathcal{L}$ and $h\in \mathcal{L}'$ if $1<p<p_\theta(1,\theta/2)$ and $h=0$ on $\partial \Omega$ if $p\ge p_\theta(1,\theta/2)$.
Consider problem {\rm (SHE)} with
\begin{equation}
\label{eq:4.2}
\mu = d(x)^\frac{\theta}{2} f(x) + h(x) \otimes \delta_1(d(x)) \in \mathcal{M}.
\end{equation}
Let $\Psi$ be a strictly increasing, nonnegative, and convex function on $[0,\infty)$.
Set
\begin{align*}
& v(x,t) := 2\Psi^{-1} ([\G(t) \Psi(f)](x)),\\
& w(x,t):= \frac{2 C_5D(x,t)}{t^{\frac{1}{2}+\frac{1}{\theta}}} \Psi^{-1} ([\K(t)\Psi(h)](x)),
\end{align*}
for $(x,t) \in Q_\infty$, where $C_5>0$ is the constant as in \eqref{eq:2.7}.
Define
\[
A(\tau) := \frac{\Psi^{-1}(\tau)^p}{\tau},\quad B_\Omega(\tau) := \frac{\tau}{\Psi^{-1}(\tau)}, \quad \mbox{for} \quad \tau>0.
\]
If
\begin{equation}
\label{eq:4.3}
\begin{split}
&\sup_{t\in(0,T)} \left(\|B(\G(t)\Psi(f))\|_{L^\infty(\Omega)}\int_0^t \|A(\G(s)\Psi(f))\|_{L^\infty(\Omega)}\,ds\right) \le \epsilon,\\	
&\sup_{t\in(0,T)} \left(\|B(\K(t)\Psi(h))\|_{L^\infty(\Omega)}\int_0^t s^{-\left(\frac{1}{2}+\frac{1}{\theta}\right)(p-1)} \|A(\K(s)\Psi(h))\|_{L^\infty(\Omega)}\,ds\right) \le \epsilon,
\end{split}
\end{equation}
for some $T\in (0,T_*)$ and a sufficiently small $\epsilon>0$, then problem {\rm (SHE)} possesses a solution $u$ in $Q_T$ such that
\[
0\le u(x,t) \le v(x,t) + w(x,t) \quad \mbox{for almost all} \quad (x,t)\in Q_T.
\]
\end{theorem}

\begin{proof}
Let $\mu$ be as in \eqref{eq:4.2}. We show that $v+w$ is a supersolution of problem {\rm (SHE)} in $Q_T$.
By  Jensen's inequality with the convexity of $\Psi$ and \eqref{eq:2.7}
we have 
\begin{equation*}
\begin{split}
\int_{\overline{\Omega}} K(x,y,t)\,d\mu(y)
&\le \int_{\Omega} G(x,y,t)f(y)\,dy + \int_{\partial \Omega} K(x,y,t)h(y)\,d\sigma(y)\\
&= [\G(t)f](x) +\frac{D(x,t)}{C_5t^{\frac{1}{2}+\frac{1}{\theta}}}[\K(t)h](x)\\
&\le \Psi^{-1}([\G(t)\Psi(f)](x)) + \frac{D(x,t)}{C_5t^{\frac{1}{2}+\frac{1}{\theta}}}\Psi^{-1}([\K(t)\Psi(h)](x))\\
& = \frac{v(x,t)+w(x,t)}{2}
\end{split}
\end{equation*}
for all $(x,t)\in Q_\infty$.
Since $(a+b)^p\le 2^{p-1}(a^p+b^p)$ for $a,b>0$, we have
\begin{equation*}
\begin{split}
&\int_{\overline{\Omega}} K(x,y,t)\,d\mu(y) +\int_0^t\G(t-s)(v(s)+w(s))^p\,ds\\
&\le  \frac{v(x,t)+w(x,t)}{2} +2^{p-1}\left[\int_0^t\G(t-s)v(s)^p\,ds + \int_0^t\G(t-s)w(s)^p\,ds\right].
\end{split}
\end{equation*}
By the semigroup property of $G$ and \eqref{eq:4.3} we see that
\begin{equation*}
\begin{split}
&\int_0^t\G(t-s)v(s)^p\,ds \\
&\le 2^p \int_0^t G(t-s) \left\|\frac{[\Psi^{-1}(\G(s)\Psi(f))]^p}{\G(s)\Psi(f)}\right\|_{L^\infty(\Omega)} G(s)\Psi(f)\,ds \\
&= 2^pG(t)\Psi(f)\int_0^t \left\|\frac{[\Psi^{-1}(\G(s)\Psi(f))]^p}{\G(s)\Psi(f)}\right\|_{L^\infty(\Omega)} \,ds\\
&\le 2^{p-1}v(t) \left\|\frac{\G(t)\Psi(f)}{\Psi^{-1}(\G(t)\Psi(f))}\right\|_{L^\infty(\Omega)}\int_0^t \left\|\frac{[\Psi^{-1}(\G(s)\Psi(f))]^p}{\G(s)\Psi(f)}\right\|_{L^\infty(\Omega)} \,ds\\
&\le C\epsilon v(t).
\end{split}
\end{equation*}
On the other hand, let $\psi\in \mathcal{L}$.
By \eqref{eq:2.7} and \eqref{eq:2.8}  we have
\begin{equation*}
\begin{split}
&\left[\G(t-s)D(\cdot,s)\K(s)\psi\right](x)\\
&= \int_{\Omega} G(x,y,t-s) D(y,s)[\K(s)\psi](y)\, dy\\
&= C_5s^{\frac{1}{2}+\frac{1}{\theta}} \int_{\Omega} G(x,y,t-s) \int_{\partial \Omega}K(y,z,s)\psi(z)\,d\sigma(z)dy\\
&= C_5s^{\frac{1}{2}+\frac{1}{\theta}}\int_{\partial \Omega} \left(\int_{\Omega} G(x,y,t-s)K(y,z,s) \,dy\right) \psi(z)\, d\sigma(z)\\
&= C_5s^{\frac{1}{2}+\frac{1}{\theta}}\int_{\partial \Omega}  K(x,z,t) \psi(z)\, d\sigma(z)\\
&= \frac{s^{\frac{1}{2}+\frac{1}{\theta}}D(x,t)}{t^{\frac{1}{2}+\frac{1}{\theta}}}[\K(t)\psi](x)
\end{split}
\end{equation*}
This together with \eqref{eq:4.3} implies that 
\begin{equation*}
\begin{split}
&\int_0^t \G(t-s)w(s)^p\,ds\\
&\le C \int_0^t s^{-\left(\frac{1}{2}+\frac{1}{\theta}\right)p}\G(t-s)  \left[D(\cdot,s)\Psi^{-1}(\K(s)\Psi(h))\right]^p\,ds\\
&\le C \int_0^t s^{-\left(\frac{1}{2}+\frac{1}{\theta}\right)p} \left\| \frac{[D(\cdot,s)\Psi^{-1}(\K(s)\Psi(h))]^p}{D(\cdot,s)\K(s)\Psi(h)}\right\|_{L^\infty(\Omega)}\\
&\qquad\qquad\qquad\qquad\times\G(t-s) D(\cdot,s)\K(s)\Psi(h) \,ds\\
&\le C\frac{D(\cdot,t)}{t^{\frac{1}{2}+\frac{1}{\theta}}}\K(t)\psi
\int_0^t s^{-\left(\frac{1}{2}+\frac{1}{\theta}\right)(p-1)} \left\| \frac{[\Psi^{-1}(\K(s)\Psi(h))]^p}{\K(s)\Psi(h)}\right\|_{L^\infty(\Omega)} \,ds\\
&\le \frac{C}{t^{\frac{1}{2}+\frac{1}{\theta}}} D(\cdot,t)\K(t)\Psi(h) \int_0^t s^{-\left(\frac{1}{2}+\frac{1}{\theta}\right)(p-1)}  \left\| \frac{[\Psi^{-1}(\K(s)\Psi(h))]^p}{\K(s)\Psi(h)}\right\|_{L^\infty(\Omega)} \,ds\\
&\le \frac{C}{t^{\frac{1}{2}+\frac{1}{\theta}}} D(\cdot,t)\Psi^{-1} (\K(t)\Psi(h))\\
&\times\left\|\frac{\K(t)\Psi(h)}{\Psi^{-1} (\K(t)\Psi(h))}\right\|_{L^\infty(\Omega)} \int_0^ts^{-\left(\frac{1}{2}+\frac{1}{\theta}\right)(p-1)}   \left\| \frac{[\Psi^{-1}(\K(s)\Psi(h))]^p}{\K(s)\Psi(h)}\right\|_{L^\infty(\Omega)} \,ds\\
&\le C\epsilon w(x,t).
\end{split}
\end{equation*}
Taking a sufficiently small $\epsilon >0$ if necessary, the above computations show that
\[
\int_{\overline{\Omega}} K(\cdot,y,t)\,d\mu(y) +\int_0^t\G(t-s)(v(s)+w(s))^p\,ds
\le v(t) + w(t) 
\]
for all $t\in(0,T)$.
This means that $v+w$ is a supersolution of problem {\rm (SHE)} in $Q_T$.
Then Lemma \ref{Lemma:2.3} implies that problem {\rm (SHE)} possesses a solution in $Q_T$.
Thus, Theorem \ref{Theorem:4.2} follows.
\end{proof}

Next, as an application of Theorem \ref{Theorem:4.2}, we obtain sufficient conditions on the solvability of problem {\rm (SHE)}.

\begin{theorem}
\label{Theorem:4.3}
Let $f\in {\mathcal L}$ and let $h\in {\mathcal L}'$ if $1<p<p_\theta(1,\theta/2)$ and $h=0$ if $p\ge p_\theta(1,\theta/2)$.
For any $q>1$, there exists $\gamma = \gamma(\Omega, N, \theta,p,q)>0$ with the following property:
if there exists $T\in (0,T_*]$ such that
\begin{equation}
\label{eq:4.4}
\begin{split}
&\sup_{z\in \overline{\Omega}} \int_{B_\Omega(z,\sigma)} D(y,\sigma^\theta) f(y)^q \,dy \le \gamma \sigma^{N-\frac{\theta q}{p-1}},\\
&\sup_{z\in \partial \Omega} \int_{B_\Omega (z,\sigma) \cap \partial\Omega} h(y)^q \,d\sigma(y) \le \gamma \sigma^{N-1 +\theta  q\left(\frac{1}{2}+\frac{1}{\theta}-\frac{1}{p-1}\right)},
\end{split}
\end{equation}
for all $\sigma \in (0,T^{1/\theta})$, then problem {\rm (SHE)} with \eqref{eq:4.2} possesses a solution in $Q_T$, with $u$ satisfying
\begin{equation}
\label{eq:4.5}
0\le u(x,t) \le 2[\G(t) f^q](x)^\frac{1}{q} + \frac{2C_5 D(x,t)}{t^{\frac{1}{2}+\frac{1}{\theta}}} [\K(t)h^q](x)^\frac{1}{q}
\end{equation}
for almost all $(x,t)\in Q_T$.
\end{theorem}

\begin{proof}
Assume \eqref{eq:4.4}. We can assume without loss of generality, that $q\in (1,p)$.
Indeed, if $q\ge p$, then, for any $1<q'<p$, we apply H\"{o}lder's inequality to obtain
\begin{equation*}
\begin{split}
&\sup_{z\in \overline{\Omega}} \int_{B_\Omega(z,\sigma)} D(y,\sigma) f(y)^{q'} \, dy \\
&\le \sup_{z\in\overline{\Omega} } \left[\int_{B_\Omega(z,\sigma)} D(y,\sigma) \,dy\right]^{1-\frac{q'}{q}} \left[\int_{B_\Omega(z,\sigma)} D(y,\sigma) f(y)^q\, dy\right]^\frac{q'}{q}\\
& \le C\gamma^{\frac{q'}{q}} \sigma^{N-\frac{\theta q'}{p-1}}
\end{split}
\end{equation*}
and
\begin{equation*}
\begin{split}
&\sup_{z\in \partial \Omega} \int_{B_\Omega(z,\sigma)\cap \partial \Omega} h(y)^{q'}\, d\sigma(y)\\
&\le \left[\int_{B_\Omega(z,\sigma)\cap \partial \Omega}\,d\sigma(y)\right]^{1-\frac{q'}{q}} 
 \left[\int_{B_\Omega(z,\sigma)\cap \partial \Omega} h(y)^{q}\right]^\frac{q'}{q}\\
  &\le C\gamma^\frac{q'}{q} \sigma^{N-1 +\theta q' \left(\frac{1}{2}+\frac{1}{\theta}-\frac{1}{p-1}\right)}
\end{split} 
\end{equation*}
for all $\sigma\in (0,T^{1/\theta})$.
Then \eqref{eq:4.4} holds with $q$ replaced by $q'$.
Furthermore, if \eqref{eq:4.5} holds for some $q'\in (1,p)$, then, since
\[
[\G(t)f^{q'}](x)^{\frac{1}{q'}} \le [\G(t)f^q](x)^\frac{1}{q}, \quad [\K(t)h^{q'}](x)^\frac{1}{q'}\le  [\K(t)h^{q}](x)^\frac{1}{q},
\]
for $x\in \Omega$ and $t>0$, the desired inequality \eqref{eq:4.5} holds.

We apply Theorem \ref{Theorem:4.2} to prove Theorem \ref{Theorem:4.3}.
Let $A$ and $B$ be as in Theorem \ref{Theorem:4.2} with $\Psi(\tau) = \tau^q$.
Then $A(\tau) =\tau^{(p/q)-1}$ and $B(\tau) = \tau^{1-(1/q)}$.
Set
\[
v(x,t) := 2[\G(t) f^q](x)^\frac{1}{q}, \quad w(x,t) := \frac{2C_5 D(x,t)}{t^{\frac{1}{2}+\frac{1}{\theta}}} [\K(t)h^q](x)^\frac{1}{q}.
\]
for all $(x,t) \in Q_T$.
It follows from \eqref{eq:2.4} that
\begin{equation*}
\begin{split}
[\G(t)f^q](x) 
&= \int_\Omega K(x,y,t) d(y)^\frac{\theta}{2} f(y)^q\,dy \\
&\le C t^{-\frac{N}{\theta}} \sup_{z\in \overline{\Omega}} \int_{B_\Omega(z,t^\frac{1}{\theta})} D(y,t) f(y)^q\,dy \le C\gamma t^{-\frac{q}{p-1}}
 \end{split} 
\end{equation*}
and
\begin{equation*}
\begin{split}
[\K(t)h^q](x) 
&=\frac{C_5 t^{\frac{1}{2}+\frac{1}{\theta}}}{D(x,t)} \int_{\partial\Omega} K(x,y,t) h(y)^q\, d\sigma(y)\\
& = \frac{C_5 t^{\frac{1}{2}+\frac{1}{\theta}}}{D(x,t)} \int_{\overline{\Omega}} K(x,y,t) h(y)^q\delta_1(d(y))\, dy\\
&\le Ct^{-\frac{N}{\theta}+\frac{1}{\theta}} \sup_{z\in \partial \Omega} \int_{B_\Omega(z,t^\frac{1}{\theta})\cap \partial \Omega} h(y)^q\,d\sigma(y)\\
& \le C\gamma t^{q\left(\frac{1}{2}+\frac{1}{\theta}-\frac{1}{p-1}\right)}
 \end{split}
\end{equation*}
for all $t\in (0,T^{1/\theta})$.
Then thanks to $q\in (1,p)$, we have
\begin{equation*}
\begin{split}
&\|B(\G(t)\Psi(f))\|_{{L^\infty(\Omega)}} \int_0^t \|A(\G(s)\Psi(f))\|_{L^\infty(\Omega)}\,ds\\
&=\|\G(t)f^q\|_{L^\infty(\Omega)}^{1-\frac{1}{q}} \int_0^t \|\G(s)f^q\|_{L^\infty(\Omega)}^{\frac{p}{q}-1}\,ds\\
&\le C\gamma^\frac{p-1}{q} t^{-\frac{q-1}{p-1}} \int_0^t s^{-\frac{p-q}{p-1}} \, ds \le C\gamma^\frac{p-1}{q}
\end{split}
\end{equation*}
for all $t\in (0,T^{1/\theta})$.
In the case of $1<p< p_\theta(1,\theta/2)$, we obtain
\begin{equation*}
\begin{split}
&\|B(\K(t)\Psi(h))\|_{L^\infty(\Omega)}\int_0^t s^{-\left(\frac{1}{2}+\frac{1}{\theta}\right)(p-1)} \|A(\K(s)\Psi(h))\|_{L^\infty(\Omega)}\,ds\\
&= \|\K(t)\Psi(h)\|_{L^\infty(\Omega)}^{1-\frac{1}{q}}\int_0^t s^{-\left(\frac{1}{2}+\frac{1}{\theta}\right)(p-1)} \|\K(s)\Psi(h)\|_{L^\infty(\Omega)}^{\frac{p}{q}-1}\,ds\\
&\le C\gamma^\frac{p-1}{q}  t^{(q-1)\left(\frac{1}{2}+\frac{1}{\theta}-\frac{1}{p-1}\right)}
\int_0^t s^{-\left(\frac{1}{2}+\frac{1}{\theta}\right)(p-1)}  s^{(p-q)\left(\frac{1}{2}+\frac{1}{\theta}-\frac{1}{p-1}\right)}\, ds\\
&\le  C\gamma^\frac{p-1}{q} 
\end{split}
\end{equation*}
for all $t\in (0,T^{1/\theta})$.
Then we apply Theorem \ref{Theorem:4.2} to obtain the desired conclusion. Thus, the proof is complete.
\end{proof}

\begin{theorem}
\label{Theorem:4.4}
Let $p=p_\theta(N,l)$ with $l\in\{0,\theta/2\}$. Let $r>0$ and set $\Phi(\tau) := \tau [\log(e+\tau)]^r$ for $\tau\ge0$. For any $T>0$, there exists $\gamma=\gamma(\Omega,N, \theta,r,T,l)>0$ such that, if $f\in{\mathcal L}$ satisfies
\[
\sup_{z\in \overline{\Omega}} \int_{B_\Omega(z,\sigma)} d(y)^l \Phi(T^\frac{1}{p-1} f(y)) \,dy \le \gamma T^{\frac{N+l}{\theta}} \left[\log\left(e+\frac{T^\frac{1}{\theta}}{\sigma}\right)\right]^{r-\frac{N+l}{\theta}}
\]
for all $\sigma\in (0,T^{1/\theta})$,
then problem {\rm (SHE)} with $\mu = d(x)^{\theta/2}f(x)$ possesses a solution $u$ in $Q_T$, with $u$ satisfying
\[
0\le u(x,t) \le C\Phi^{-1}\left([\G(t)\Phi(T^\frac{1}{p-1})f](x)\right)
\]
for almost all $(x,t)\in Q_T$ for some $C>0$.
\end{theorem} 

\begin{proof}
Let $0<\epsilon <p-1$. We find $L\in[e,\infty)$ with the following properties:
\begin{itemize}
\item[(a)] $\Psi (s) := s[\log(L + s)]^r$ is positive and convex in $(0,\infty)$;
\item[(b)] $s^p/\Psi(s)$ is increasing in $(0,\infty)$;
\item[(c)] $s^\epsilon [\log(L+s)]^{-pr}$ is increasing in $(0,\infty)$.
\end{itemize} 
Since $C^{-1}\Phi(s) \le \Psi(s) \le C\Phi(s)$ for $s\in (0,\infty)$, we see that
\begin{equation}
\label{eq:4.6}
\sup_{z\in \overline{\Omega}} \int_{B_\Omega(z,\sigma)} d(y)^l \Psi(T^\frac{1}{p-1} f(y)) \,dy \le \gamma T^{\frac{N+l}{\theta}} \left[\log\left(e+\frac{T^\frac{1}{\theta}}{\sigma}\right)\right]^{r-\frac{N+l}{\theta}}
\end{equation}
for all $\sigma\in (0,T^{1/\theta})$. Here we can assume, without loss of generality, that $\gamma\in(0,1)$.
Set 
\[
z(x,t) := \left[\G(t)\Psi (T^\frac{1}{p-1}f)\right](x) = \int_\Omega K(x,y,t) d(y)^\frac{\theta}{2} \Psi(T^\frac{1}{p-1}f(y))\, dy.
\]
By \eqref{eq:2.4} we have
\begin{equation*}
\begin{split}
\|z(t)\|_{L^\infty(\Omega)} 
&\le C t^{-\frac{N}{\theta}} \sup_{z\in \overline{\Omega}} \int_{B_\Omega(z,t^\frac{1}{\theta})} D(y,t) \Psi(T^\frac{1}{p-1}f(y))\,dy\\
&\le Ct^{-\frac{N+l}{\theta}} \sup_{z\in \overline{\Omega}} \int_{B_\Omega(z,t^\frac{1}{\theta})} d(y)^l \Psi(T^\frac{1}{p-1}f(y))\,dy\\
&\le C\gamma {t_T}^{-\frac{N+l}{\theta} } |\log t_T|^{r-\frac{N+l}{\theta}}
\le C{t_T}^{-\frac{N+l}{\theta} } |\log t_T|^{r-\frac{N+l}{\theta}}
\end{split}
\end{equation*}
for all $t\in (0,T)$, where $t_T := t/(2T) \in (0,1/2)$.
Since 
\[
C^{-1} \tau [\log(L+\tau)]^{-r} \le \Psi^{-1}(\tau) \le C \tau [\log(L+\tau)]^{-r}
\]
for $\tau>0$, we have
\begin{equation*}
\begin{split}
&A(z(x,t)) =\frac{\Psi^{-1}(z(x,t))^p}{z(x,t)}\le Cz(x,t)^{p-1}[\log(L+z(x,t))]^{-pr},\\
&B(z(x,t)) =\frac{z(x,t)}{\Psi^{-1}(z(x,t))} \le C[\log(L+z(x,t))]^r,
\end{split}
\end{equation*}
for $(x,t)\in Q_\infty$.
Then we have 
\begin{equation*}
\begin{split}
0
&\le A(z(x,t)) \le C \|z(t)\|_{L^\infty(\Omega)}^{p-1-\epsilon}  \|z(t)\|_{L^\infty(\Omega)}^\epsilon [\log(L + \|z(t)\|_{L^\infty(\Omega)})]^{-pr}\\
&\le C\gamma^{p-1-\epsilon} t_T^{-\frac{N+l}{2}(p-1)}|\log t_T|^{\left(r-\frac{N+l}{\theta}\right)(p-1)}|\log t_T|^{-pr}\\
& = C\gamma^{p-1-\epsilon} t_T^{-1}|\log t_T|^{-r-1}
\end{split}
\end{equation*}
and
\begin{equation*}
\begin{split}
0\le B(z(x,t)) \le C[\log(L + \|z(t)\|_{L^\infty(\Omega)})]^r \le C |\log t_T|^r
\end{split}
\end{equation*}
for all $(x,t)\in Q_T$, where $C$ is independent of $\gamma$. Hence 
\begin{equation*}
\begin{split}
&\|B(z(t))\|_{L^\infty(\Omega)} \int_0^t\|A(z(s))\|_{L^\infty(\Omega)} \, ds \\
&\le C\gamma^{p-1-\epsilon}|\log t_T|^r \int_0^t s^{-1} |\log s_T|^{-r-1} \, ds\\
&=  C\gamma^{p-1-\epsilon}|\log t_T|^r \int_0^t \frac{2T}{s}\left[-\log\frac{s}{2T}\right]^{-r-1}\,ds\\
&= CT \gamma^{p-1-\epsilon}
\end{split}
\end{equation*}
for all $t\in (0,T)$. Therefore, if $\gamma>0$ is small enough, the we apply Theorem \ref{Theorem:4.2} to find a solution $u$ of problem {\rm (SHE)} in $Q_T$ such that
\[
0\le u(x,t) \le 2\Psi^{-1} (z(x,t)) \le C\Phi([\G(t)\Phi(f)](x))
\]
for almost all $(x,t)\in Q_T$.
Thus, Theorem \ref{Theorem:4.4} follows.
\end{proof}

\begin{theorem}
\label{Theorem:4.5}
Let $p=p_\theta(N,\theta/2)<p_\theta(1,\theta/2)$. Let $r>0$ and $\Phi(\tau):= \tau [\log (e+\tau)]^r$ for $\tau\ge0$.
For any $T>0$, there exists $\gamma= \gamma(\Omega,N,\theta, r,T)>0$ such that, if $h\in{\mathcal L}'$ satisfies
\begin{equation}
\label{eq:4.7}
\sup_{z\in \partial \Omega} \int_{B_\Omega (z,\sigma)\cap \partial\Omega} \Phi (T^\frac{1}{p-1} h(y)) \, d\sigma(y) \le \gamma T^{\frac{N-1}{\theta}} \left[\log\left(e+\frac{T^\frac{1}{\theta}}{\sigma}\right)\right]^{r-\frac{2N+\theta}{2\theta}}
\end{equation}
for all $\sigma\in (0,T^{1/\theta})$, then problem {\rm (SHE)} with $\mu = h(x) \otimes \delta_1(d(x))$ possesses a solution $u$ in $Q_T$, with $u$ satisfying
\[
0\le u(x,t) \le C\frac{D(x,t)}{t^{\frac{1}{2}+\frac{1}{\theta}}} \Psi^{-1} ([\K(t)\Psi(h)](x))
\]
for almost all $(x,t)\in Q_T$, for some $C>0$.
\end{theorem}
\begin{proof}
Assume \eqref{eq:4.7}. We can assume, without loss of generality, that $\gamma\in(0,1)$.
Define $\Psi$ as in the proof of Theorem \ref{Theorem:4.4}. Set
\[
z(x,t) := \left[\K (t) \Psi(T^\frac{1}{p-1}h)\right] (x).
\]
By Lemma \ref{Lemma:2.1} and \eqref{eq:4.7} we see that
\begin{equation*}
\begin{split}
\|z(t)\|_{L^\infty(\Omega)}
&\le Ct^{-\frac{N-1}{\theta}}  \sup_{z\in \partial\Omega} \int_{B_\Omega\cap \partial \Omega} \Psi(T^\frac{1}{p-1} h(y))\, d\sigma(y)\\
&\le C\gamma t_T^{-\frac{N-1}{\theta}} |\log t_T|^{r-\frac{2N+\theta}{2\theta}} \le Ct_T^{-\frac{N-1}{\theta}} |\log t_T|^{r-\frac{2N+\theta}{2\theta}} 
\end{split}
\end{equation*}
for $t\in(0,T)$, where $t_T:= t/(2T) \in (0,1/2)$.
By the same argument as in the proof of Theorem \ref{Theorem:4.4} we have 
\begin{equation*}
\begin{split}
& 0 \le A(z(x,t)) \le C\gamma^{p-1-\epsilon} t^{-\frac{N-1}{\theta}(p-1)} |\log t_T|^{\left(r-\frac{2N+\theta}{2\theta}\right)(p-1) -pr}\\
& = C\gamma^{p-1-\epsilon} t_T^{-\frac{2N-2}{2N+\theta}} |\log t_T|^{-r-1},\\
&0\le B(z(x,t)) \le C|\log t_T|^r,
\end{split}
\end{equation*}
for all $(x,t)\in Q_T$, where $C$ is independent of $\gamma$. It follows that
\begin{equation*}
\begin{split}
&\|B(z(t))\|_{L^\infty(\Omega)} \int_0^t s^{-\left(\frac{1}{2}+\frac{1}{\theta}\right)(p-1)} \|A(z(t))\|_{L^\infty(\Omega)}\,ds\\
&= (2T)^{-\left(\frac{1}{2}+\frac{1}{\theta}\right)(p-1)}\|B(z(t))\|_{L^\infty(\Omega)} \int_0^t s_T^{-\left(\frac{1}{2}+\frac{1}{\theta}\right)(p-1)} \|A(z(t))\|_{L^\infty(\Omega)}\,ds\\
&\le CT^{-\left(\frac{1}{2}+\frac{1}{\theta}\right)(p-1)} \gamma^{p-1-\epsilon} |\log t_T|^r \int_0^t s_T^{-1} |\log s_T|^{-r-1} \,ds \\
&\le  CT^{1-\left(\frac{1}{2}+\frac{1}{\theta}\right)(p-1)} \gamma^{p-1-\epsilon}.
 \end{split}
\end{equation*}
Thus, Theorem \ref{Theorem:4.2} leads to the desired conclusion. The proof is complete.
\end{proof}

\begin{remark}
In Theorems \ref{Theorem:4.2}, \ref{Theorem:4.3}, and \ref{Theorem:4.5},
we assume $p<p_\theta(1,\theta/2)$ when we consider $h(x)$ as in \eqref{eq:4.2}.
This assumption is probably essential and we would expect the following assertion to hold:
\begin{itemize}
\item Let $p\ge p_\theta(1,\theta/2)$.
If problem {\rm (SHE)} possesses a local-in-time solution,
 then $\mu (\partial\Omega)=0$ must hold.
\end{itemize}
Actually, in the case where $\theta=2$ and $\Omega= \mathbb{R}^N_+$,
the above assertion holds (see \cite{HIT2}).
\end{remark}

\section{Proofs of Theorems \ref{Theorem:1.1} and \ref{Theorem:1.2}.}
In this section by applying the necessary conditions and the sufficient conditions proved in previous sections,
we prove Theorems \ref{Theorem:1.1} and \ref{Theorem:1.2}.
\begin{proof}[Proof of Theorem \ref{Theorem:1.1}.]
We first prove assertion (i).
Let $p<p_\theta(N,0)$ and $\nu \in {\mathcal M}$.
Since
\[
1-\frac{N}{\theta}(p-1) >0,
\] 
we have 
\begin{equation*}
\begin{split}
&\int_0^Ts^{-\frac{N}{\theta}(p-1)} \left(\sup_{z\in \overline{\Omega}} \int_{B_\Omega(z,s^\frac{1}{\theta})}\frac{d\mu(y)}{\dist(y)^\frac{\theta}{2}+\sqrt{s}}\right)^{p-1}\,ds \\
& = \int_0^T s^{-\frac{N}{\theta}(p-1)} \left(\sup_{z\in \overline{\Omega}} \int_{B_\Omega(z,T^\frac{1}{\theta})}D(y,s) \, d\nu(y)\right)^{p-1}\,ds\\
&\le \left[\sup_{z\in \overline{\Omega}}\nu(B_\Omega(z,T^\frac{1}{\theta}))\right]^{p-1} \int_0^T s^{-\frac{N}{\theta}(p-1)}\,ds \\
&\le C\left[\sup_{z\in \overline{\Omega}}\nu(B_\Omega(z,T^\frac{1}{\theta}))\right]^{p-1} T^{1-\frac{N}{\theta}(p-1)}
\end{split}
\end{equation*}
for $T>0$.
Taking sufficient small $T>0$ if necessary, we see that \eqref{eq:4.1}  holds.
It follows from Theorem \ref{Theorem:4.1} that assertion (i) follows.

We prove assertion (ii).
Let $z\in \Omega$, $\kappa>0$, and $\mu = \kappa d(x)^{\theta/2} \varphi_z(x)$ in ${\mathcal M}$.
If $p>p_\theta(N,0)$, then we find $q>1$ such that
\begin{equation*}
\sup_{x\in \overline{\Omega}} \int_{B_\Omega(x,\sigma)} D(y,\sigma^\theta) (\kappa \varphi_z(y))^q \,dy
\le \kappa^q \int_{B(z,\sigma)} |y-z|^{-\frac{2q}{p-1}} \, dy \le C\kappa^q \sigma^{N-\frac{\theta q}{p-1}}
\end{equation*}
for all $\sigma\in (0,1)$.
If $p=p_\theta(N,0)$, then for any $r\in (0, N/\theta)$, we have
\begin{equation*}
\begin{split}
&\sup_{x\in \overline{\Omega}} \int_{B_\Omega(x,\sigma)} \kappa \varphi_z(y)[\log(e+ \kappa \varphi_z(y))]^r \, dy\\
&\le C\kappa \int_{B(z,\sigma)}|y-z|^{-N}|\log|y-z||^{-\frac{N}{\theta}-1+r}\, dy
\le C\kappa |\log\sigma|^{-\frac{N}{\theta}+r}\\
\end{split}
\end{equation*}
for all small enough $\sigma>0$ and $\kappa\in(0,1)$. Then, if $\kappa>0$ is small enough, by Theorem \ref{Theorem:4.3} with $p>p_\theta(N,0)$ and Theorem \ref{Theorem:4.4} with $l=0$ we find a local-in-time solution of problem {\rm (SHE)}. 
Then we obtain the desired conclusion and assertion (ii) follows.

Finally, we prove assertion (iii).
Assume that problem {\rm (SHE)} possesses a local-in-time solution. By Theorem \ref{Theorem:3.1} we have
\begin{equation}
\label{eq:5.1}
\begin{split}
\kappa \int_{B_\Omega(z,\sigma)} d(y)^\frac{\theta}{2} \varphi_z(y) \, dy 
&\le C \sigma^{-\frac{2}{p-1}} \int_{B_\Omega(z,\sigma)} d(y)^\frac{\theta}{2} \,dy\\
& \le Cd(z)^\frac{\theta}{2} \sigma^{N-\frac{\theta}{p-1}}
\end{split}
\end{equation}
for all small enough $\sigma>0$.
Furthermore, if $p=p_\theta(N,0)$, then 
\begin{equation}
\label{eq:5.2}
\kappa \int_{B_\Omega (z,\sigma)} d(y)^\frac{\theta}{2} \varphi_z(y) \, dy \le Cd(z)^\frac{\theta}{2} |\log \sigma|^{-\frac{N}{2}}
\end{equation}
for all small $\sigma>0$. On the other hand, it follows that
\[
\int_{B_\Omega(z,\sigma)} d(y)^\frac{\theta}{2} \varphi_z(y) \,dy \ge 
\left\{
\begin{array}{ll}
	Cd(z)^\frac{\theta}{2}\sigma^{N-\frac{\theta}{p-1}}, \quad &\mbox{if}\quad p>p_\theta(N,\theta/2),\vspace{3pt}\\
	Cd(z)^\frac{\theta}{2}|\log \sigma|^{-\frac{N}{\theta}}, \quad   &\mbox{if}\quad p=p_\theta(N,\theta/2). \vspace{3pt}
\end{array}
\right.
\]
This together with \eqref{eq:5.1} and \eqref{eq:5.2} implies that \eqref{Thm:3.1.1} and \eqref{Thm:3.1.2} do not hold for sufficiently large $\kappa>0$ and $\kappa_z$ is uniformly bounded on $\Omega$.
Thus, Theorem \ref{Theorem:1.1} follows and the proof is complete.
\end{proof}

\begin{proof}[Proof of Theorem \ref{Theorem:1.2}.]
We first prove assertion (i).
Let $p<p_\theta(N,\theta/2)$ and $\mu \in {\mathcal M}$.
Since
\[
1-\frac{2N+\theta}{2\theta}(p-1) >0,
\] 
we have
\begin{equation*}
\begin{split}
&\int_0^Ts^{-\frac{N}{\theta}(p-1)} \left(\sup_{z\in \overline{\Omega}} \int_{B_\Omega(z,s^\frac{1}{\theta})}\frac{d\mu(y)}{\dist(y)^\frac{\theta}{2}+\sqrt{s}}\right)^{p-1}\,ds \\
&\le \left[\sup_{z\in \overline{\Omega}}\mu(z,T^\frac{1}{\theta})\right]^{p-1} \int_0^T s^{-\frac{2N+\theta}{2\theta}(p-1)}\,ds\\
&\le C \left[\sup_{z\in \overline{\Omega}}\mu(z,T^\frac{1}{\theta})\right]^{p-1} T^{1-\frac{2N+\theta}{2\theta}(p-1)}
\end{split}
\end{equation*}
for $T>0$.
Taking sufficient small $T>0$ if necessary, we see that \eqref{eq:4.1}  holds.
It follows from Theorem \ref{Theorem:4.1} that assertion (i) follows.

We prove assertion (ii).
Let $z\in \partial\Omega$, $\kappa>0$, and $\mu = \kappa d(x)^{\theta/2} \varphi_z(x)$ in ${\mathcal M}$.
If $p>p_\theta(N,\theta/2)$, then we find $q>1$ such that
\begin{equation*}
\begin{split}
\int_{B_\Omega(x,\sigma)} D(y,\sigma^\theta) (\kappa \varphi_z(y))^q \,dy
&\le \kappa^q \sigma^{-\frac{\theta}{2}}\int_{B(z,3\sigma)} d(y)^\frac{\theta}{2} |y-z|^{-\frac{2q}{p-1}} \, dy \\&\le C\kappa^q \sigma^{N-\frac{\theta q}{p-1}}
\end{split}
\end{equation*}
for all $x\in B_\Omega (z,2\sigma)$ and $\sigma\in (0,1)$.
Furthermore, we have 
\begin{equation*}
\begin{split}
\int_{B_\Omega(x,\sigma)} D(y,\sigma^\theta) (\kappa \varphi_z(y))^q \,dy
&\le \kappa^q \sigma^{-\frac{\theta}{2}}\int_{B(x,\sigma)} |y-z|^{-\frac{2q}{p-1}} \, dy \\
&\le C\kappa^q \sigma^N |x|^{-\frac{\theta q}{p-1}}\le C\kappa^q \sigma^{N-\frac{\theta q}{p-1}}
\end{split}
\end{equation*}
for all $x\in \overline{\Omega}\setminus B_\Omega (z,2\sigma)$ and $\sigma\in (0,1)$.
These imply that
\[
\sup_{x\in\overline{\Omega}}\int_{B_\Omega(x,\sigma)} D(y,\sigma^\theta) (\kappa \varphi_z(y))^q \,dy \le C\kappa^q \sigma^{N-\frac{\theta q}{p-1}}
\]
for all $\sigma\in (0,1)$ if $p>p_\theta(N,\theta/2)$.
If $p=p_\theta(N,\theta/2)$, then for any $r\in (0, N/\theta)$, we have
\begin{equation*}
\begin{split}
&\sup_{x\in \overline{\Omega}} \int_{B_\Omega(x,\sigma)} \kappa d(y)^\frac{\theta}{2} \varphi_z(y)[\log(e+ \kappa \varphi_z(y))]^r \, dy\\
&\le C\kappa \int_{B(z,\sigma)}|y-z|^{-N}|\log|y-z||^{-\frac{N}{\theta}-1+r} \, dy
\le C\kappa |\log\sigma|^{-\frac{N}{\theta}+r}\\
\end{split}
\end{equation*}
for all small enough $\sigma>0$ and $\kappa\in(0,1)$. Then, if $\kappa>0$ is small enough, by Theorem \ref{Theorem:4.3} with $p>p_\theta(N,\theta/2)$ and Theorem \ref{Theorem:4.4} with $l=\theta/2$ we find a local-in-time solution of problem {\rm (SHE)}. 
Then we obtain the desired conclusion and assertion (ii) follows.

Finally, we prove assertion (iii).
Assume that problem {\rm (SHE)} possesses a local-in-time solution. By Theorem \ref{Theorem:3.1} we have
\begin{equation}
\label{eq:5.3}
\begin{split}
\kappa \int_{B_\Omega(z,\sigma)} d(y)^\frac{\theta}{2} \varphi_z(y) \, dy 
&\le C \sigma^{-\frac{2}{p-1}} \int_{B_\Omega(z,\sigma)} d(y)^\frac{\theta}{2} \,dy\\
& \le C \sigma^{N+\frac{\theta}{2}-\frac{\theta}{p-1}}
\end{split}
\end{equation}
for all small enough $\sigma>0$.
Furthermore, if $p=p_\theta(N,\theta/2)$, then 
\begin{equation}
\label{eq:5.4}
\kappa \int_{B_\Omega (z,\sigma)} d(y)^\frac{\theta}{2} \varphi_z(y) \, dy \le C |\log \sigma|^{-\frac{2N+\theta}{2\theta}}
\end{equation}
for all small $\sigma>0$. On the other hand, it follows that
\begin{equation}
\label{eq:5.5}
\int_{B_\Omega(z,\sigma)} d(y)^\frac{\theta}{2} \varphi_z(y) \,dy \ge 
\left\{
\begin{array}{ll}
	C\sigma^{N+\frac{\theta}{2}-\frac{\theta}{p-1}}, \quad &\mbox{if}\quad p>p_\theta(N,\theta/2),\vspace{3pt}\\
	C|\log \sigma|^{-\frac{2N+\theta}{2\theta}}, \quad   &\mbox{if}\quad p=p_\theta(N,\theta/2). \vspace{3pt}
\end{array}
\right.
\end{equation}
By \eqref{eq:5.3}, \eqref{eq:5.4}, and \eqref{eq:5.5}
we see that \eqref{Thm:3.1.1} and \eqref{Thm:3.1.2} do not hold for sufficiently large $\kappa>0$ and $\kappa_z\le C$.
Thus, Theorem \ref{Theorem:1.2} follows and the proof is complete.
\end{proof}


\end{document}